\documentclass[11pt,a4paper]{article}

\usepackage[margin=1truein]{geometry}

\usepackage{lmodern}
\usepackage[T1]{fontenc}
\usepackage{textcomp}
\usepackage{bm}

\usepackage[group-separator={,}]{siunitx}
\usepackage{amsmath}
\usepackage{amssymb}
\usepackage{amsthm}
\usepackage{subcaption}
\usepackage{physics}
\usepackage{mathtools}

\DeclarePairedDelimiterX{\Set}[2]{\{}{\}}{\,{#1}\;\delimsize|\;{#2}\,}

\usepackage{mleftright}\mleftright

\usepackage{nicematrix}

\usepackage{algorithm}
\usepackage[noend]{algpseudocode}

\algnewcommand{\algorithmicinput}{\textbf{Input:}}
\algnewcommand{\Input}{\item[\algorithmicinput]}
\algnewcommand{\algorithmicoutput}{\textbf{Output:}}
\algnewcommand{\Output}{\item[\algorithmicoutput]}
\algnewcommand{\algorithmicbreak}{\textbf{break}}
\algnewcommand{\Break}{\State{\algorithmicbreak}}
\algnewcommand{\nil}{\textbf{nil}}

\usepackage{etoolbox}
\cslet{blx@noerroretextools}\empty%
\usepackage[
    date=year,
    isbn=false,
    natbib,
    url=false,
    sorting=nyt,
    style=trad-abbrv
]{biblatex}

\usepackage{booktabs}
\usepackage{multirow}

\usepackage{enumitem}
\setlist{font=\upshape,leftmargin=*}
\setlist[1]{labelindent=\parindent}
\setlist[enumerate,1]{label={(\arabic*)}}

\usepackage[
    bookmarksnumbered,
    colorlinks,
    hypertexnames=false,
    pdfdisplaydoctitle,
    pdfusetitle,
    unicode
]{hyperref}
\usepackage{url}

\usepackage{upref}
\usepackage[capitalize,noabbrev]{cleveref}

\expandafter\def\csname ver@etex.sty\endcsname{3000/12/31}
\usepackage{autonum}

\newtheorem{theorem}{Theorem}[section]
\newtheorem{lemma}[theorem]{Lemma}
\newtheorem{proposition}[theorem]{Proposition}
\newtheorem{corollary}[theorem]{Corollary}
\newtheorem{claim}[theorem]{Claim}
\theoremstyle{definition}

\newtheorem{remark}[theorem]{Remark}

\ifpdf
\hypersetup{
  pdftitle={Algebraic Algorithms for Fractional Linear Matroid Parity via Non-commutative Rank},
  pdfauthor={T. Oki and T. Soma}
}
\fi

\title{%
  Algebraic Algorithms for Fractional Linear Matroid Parity\\
  via Non-commutative Rank\thanks{
    A preliminary version of this paper appeared at the 34th Annual ACM-SIAM Symposium on Discrete Algorithms (SODA '23).
  }
}

\author{%
  Taihei Oki\thanks{
    Graduate School of Information Science and Technology, the University of Tokyo, Tokyo 113-8656, Japan. Current address: Institute for Chemical Reaction Design and Discovery (ICReDD), Hokkaido University, Sapporo 001-0021, Japan. E-mail: \url{oki@icredd.hokudai.ac.jp}
  }%
  \and Tasuku Soma\thanks{
    Department of Mathematics, Massachusetts Institute of Technology, Cambridge, MA 02139. Current address: The Instutite of Statistical Mathematics, Tokyo 190-8562, Japan. E-mail: \url{soma@ism.ac.jp}
  }
}

\newcommand{\caA}{\mathcal{A}}
\newcommand{\caB}{\mathcal{B}}

\newcommand{\N}{\mathbb{N}}

\newcommand{\Q}{\mathbb{Q}}
\newcommand{\R}{\mathbb{R}}
\newcommand{\K}{\mathbb{K}}
\newcommand{\F}{\mathbb{F}}
\newcommand{\bfzero}{\mathbf{0}}
\newcommand{\ones}{\mathbf{1}}
\newcommand{\blowup}[2]{{#1}^{\{#2\}}}
\newcommand{\zho}{{\bigl\{0, \frac12, 1\bigr\}}}
\newcommand{\CLVtime}{m^4n^\omega}
\newcommand{\GPtime}{m^4n^{\omega+3}}

\DeclareMathOperator{\ncrank}{nc-rank}

\DeclareMathOperator{\sgn}{sgn}

\DeclareMathOperator{\GL}{GL}

\DeclareMathOperator{\supp}{supp}
\DeclareMathOperator{\pf}{pf}
\DeclareMathOperator{\im}{Im}
\DeclarePairedDelimiter{\floor}{\lfloor}{\rfloor}
\DeclarePairedDelimiter{\agbr}{\langle}{\rangle}

\newcommand{\fsbr}[1]{\text{\textup{$\mathrm{<}\mkern-2.3mu\llap(\,#1\mathrm{>}\mkern-4.6mu\llap)$}}\,}

\usepackage{xcolor}

\newcommand{\modify}[1]{#1}

\begin{document}

\maketitle

\begin{abstract}
Matrix representations are a powerful tool for designing efficient algorithms for combinatorial optimization problems such as matching, and linear matroid intersection and parity.
In this paper, we initiate the study of matrix representations using the concept of non-commutative rank (nc-rank), which has recently attracted attention in the research of Edmonds' problem.
We reveal that the nc-rank of the matrix representation of linear matroid parity corresponds to the optimal value of fractional linear matroid parity: a half-integral relaxation of linear matroid parity.
Based on our representation, we present an algebraic algorithm for the fractional linear matroid parity problem by building a new technique to incorporate the search-to-decision reduction into the half-integral problem represented via the nc-rank.
We further present a faster divide-and-conquer algorithm for finding a maximum fractional matroid matching and an algebraic algorithm for finding a dual optimal solution.
They together lead to an algebraic algorithm for the weighted fractional linear matroid parity problem.
Our algorithms are significantly simpler and faster than the existing algorithms.
\end{abstract}

\section{Introduction}\label{sec:introduction}
Matrix representations of combinatorial optimization problems have been a powerful tool for designing efficient algorithms.
This line of research originates in Tutte's work~\cite{Tutte1947} for the matching problem, where his matrix representation is now known as the \emph{Tutte matrix}.
Edmonds~\cite{Edmonds1967} dealt with a simpler representation for the bipartite case, and this result was later extended by Tomizawa and Iri~\cite{Tomizawa1974} to the linear matroid intersection problem.
Unifying these works, Lovász~\cite{Lovasz1979} gave a matrix representation for the \emph{linear matroid parity problem}~\cite{Lawler1976} (also called the \emph{linear matroid matching problem}), a common generalization of the matching and linear matroid intersection problems.
This problem is to find a linearly independent subset of $L = \{\ell_1, \dotsc, \ell_m\}$ with maximum cardinality, where each $\ell_i \subseteq \K^n$ is a given two-dimensional vector subspace over a field $\K$, called a \emph{line}.

These matrix representations are of the form \modify{of} so called \emph{linear} (\emph{symbolic}) \emph{matrices}:
\begin{align}\label{def:linear}
    A = \sum_{i=1}^m x_i A_i,
\end{align}
where $x_1, \dotsc, x_m$ are distinct indeterminates (symbols), and $A_1, \dotsc, A_m$ are constant matrices over a field $\K$ determined from a given instance of the problem.
Intuitively, each $x_i$ corresponds to \modify{an} element of the problem's ground set, and setting $x_i$ to zero means removing the element from consideration.
Every $A_i$ is (i) a matrix having only one nonzero entry for the bipartite matching problem, (ii) a skew-symmetric matrix having two nonzero entries for the matching problem, (iii) a rank-one matrix for the linear matroid intersection problem, and (iv) a rank-two skew-symmetric matrix for the linear matroid parity problem.
The rank of $A$ (as a matrix over the rational function field $\K(x_1, \dotsc, x_m)$) coincides with the size of maximum bipartite matchings and maximum common independent sets, and with twice the size of maximum matching and maximum independent parity set.

The problem of computing the rank of a given linear matrix is called \emph{Edmonds' problem}~\cite{Edmonds1967}.
When $|\K|$ is large enough, we can solve Edmonds' problem by substituting random values drawn from $\K$ into the indeterminates $x_1, \dotsc, x_m$ and computing the rank of the obtained constant matrix; the probability of success can be bounded by the Schwartz--Zippel lemma~\cite{Zippel1979,Schwartz1980}.
This remarkably simple idea leads to efficient randomized polynomial-time algorithms for various combinatorial optimization problems via matrix representations.
These algorithms are called \emph{algebraic algorithms}.
Indeed, the current fastest algorithm for linear matroid intersection~\cite{Harvey2009} and parity~\cite{Cheung2014} are algebraic algorithms.

A major open question in Edmonds' problem is to develop a \emph{deterministic} polynomial-time algorithm for general $A_i$; the existence of such an algorithm would imply non-trivial circuit complexity lower bounds~\cite{Kabanets2004}.
To shed light on this question, recent studies~\cite{Hrubes2015,Ivanyos2018,Garg2020,Hamada2021} focused on the \emph{non-commutative} version of Edmonds' problem.
This problem is to compute the \emph{non-commutative rank} (nc-rank) of a linear matrix: the rank when the indeterminates $x_1, \dotsc, x_m$ are regarded as ``pairwise non-commutative'', i.e., $x_i x_j \ne x_j x_i$ if $i \ne j$.
More precisely, the nc-rank of $A$, denoted as $\ncrank A$, is the inner rank as a matrix over the free ring $\K\agbr{x_1, \dotsc, x_m}$ generated by non-commutative indeterminates $x_1, \dotsc, x_m$, and is the rank as a matrix over the free skew field $\K\fsbr{x_1, \dotsc, x_m}$~\cite{Cohn1985}; see~\cite{Cohn1995,Fortin2004,Garg2020}.
An equivalent and elementary definition of nc-rank involves the \emph{blow-up} of linear matrices.
For $d \ge 1$, the $d$th-order \emph{blow-up} of an $n \times n$ linear matrix $A$ is a $dn \times dn$ linear matrix defined by
\begin{align}\label{def:blow-up}
    \blowup{A}{d} = \sum_{i=1}^m X_i \otimes A_i,
\end{align}
where $X_1, \dotsc, X_m$ are $d \times d$ matrices of distinct indeterminates in their entries, and $\otimes$ denotes the Kronecker product.
Then, $\ncrank A$ is equal to $\frac1d \rank \blowup{A}{d}$ for $d \ge n-1$~\cite{Derksen2017}.
The nc-rank can also be defined via a min-max type formulation~\cite{Fortin2004}.
Recent breakthrough results~\cite{Garg2020,Ivanyos2018,Hamada2021} show that non-commutative Edmonds' problem is solvable in deterministic polynomial time.

Given the recent advances in the studies of nc-rank, it is quite natural to ask:
\emph{Can we devise an efficient and simple randomized algorithm for combinatorial optimization problems via nc-rank?}
However, to the best of our knowledge, no connection between nc-rank and natural combinatorial optimization problems is known, which is in contrast to the aforementioned matching, linear matroid intersection, and linear matroid parity problems.

\subsection{Contributions}
In this paper, we initiate the study of algebraic algorithms via nc-rank.
Our contributions are summarized as follows.

\subsubsection{Matrix Representation of Fractional Linear Matroid Parity via Nc-rank}
We first reveal that the nc-rank of the matrix representation of the linear matroid parity problem is equal to twice the optimal value of the corresponding \emph{fractional linear matroid parity problem} (also known as the \emph{fractional linear matroid matching problem}).
This problem, introduced by Vande Vate~\cite{VandeVate1992} \modify{and Chang et al.~\cite{Chang2001a,Chang2001b}}, is a continuous relaxation of the linear matroid parity problem formally defined as follows.\footnote{\modify{
    Vande Vate~\cite{VandeVate1992} introduced the fractional matroid parity problem for general finite matroids, and subsequently, Chang et al.~\cite{Chang2001a,Chang2001b} dealt with a more general setting including matroids with infinite ground sets yet finite ranks.
    The fractional linear matroid parity problem is a special case when the matroid is the linear matroid whose ground set is all $n$-dimensional nonzero vectors over $\K$.
}}
Given lines $\ell_1, \dots, \ell_m$ as two dimensional subspaces in $\K^n$, a nonnegative vector $y = (y_1, \dotsc, y_m) \in \R^m$ is called a \emph{fractional matroid matching} if $\sum_{i=1}^m \dim(\ell_i \cap S) y_i \leq \dim S$ for any subspace $S$ in $\K^n$.
The set of all fractional matroid matchings is called the \emph{fractional matroid parity polytope} and it is a half-integral polytope contained in ${[0, 1]}^m$ whose integral points correspond to the feasible solutions of the linear matroid parity problem.
The fractional matroid parity problem is to find a fractional matroid matching maximizing $|y| \coloneqq \sum_{i=1}^m y_i$.
The dual problem is called the \emph{minimum 2-cover problem}.
A combinatorial algorithm for finding a maximum fractional matroid matching as well as a minimum 2-cover was given by Chang et al.~\cite{Chang2001a,Chang2001b}; see \Cref{sec:fractional-linear-matroid-parity} for detail.
We prove our claim by establishing a correspondence between dual solutions of the fractional linear matroid parity and non-commutative Edmonds' problems.

\subsubsection{Algebraic Algorithm for Fractional Linear Matroid Parity}
\paragraph{Simple algebraic algorithm.}
The above result provides a ``matrix representation'' that involves the nc-rank for the fractional linear matroid parity problem.
When $|\K|$ is large, one can compute the nc-rank of an $n \times n$ linear matrix $A$ by substituting random values into the entries of $X_1, \dotsc, X_m$ in $\blowup{A}{n-1}$ with high probability.
Hence, our matrix representation immediately yields an algebraic algorithm to compute the \emph{size} of a maximum fractional matroid matching.
This algorithm, however, does not output an actual maximum fractional matroid matching.
In the known matrix representations of matching and linear matroid intersection and parity, this issue can be addressed with the \emph{search-to-decision} reduction.
We can test whether each element, say $i$, is contained in an optimal solution by simply checking whether setting $x_i = 0$ decreases the rank.
Thus, one can also construct an optimal solution to these problems.
But this is possible because these problems have $\{0,1\}$-optimal solutions.
Since the fractional linear matroid parity problem may have a half-integral solution, it is unclear how to apply this reduction to our nc-rank-based matrix representation.

To this end, we establish a novel correspondence between the rank of $X_i$ in the blow-up and the $i$th component of a fractional matroid matching.
Letting $A$ be the matrix representation of this problem, we show that $\ncrank A = \frac12 \rank \blowup{A}{2}$ holds, i.e., the second-order blow-up is enough for attaining the nc-rank.
Note that every $X_i$ in $\blowup{A}{2}$ is of size $2 \times 2$.
Roughly speaking, we show that restricting the rank of $X_i$ to $0$, $1$, or $2$ corresponds to setting an upper bound on the $i$th component $y_i$ of a solution $y \in \R^m$ to $0$, $\frac12$, or $1$, respectively.
Our proof is based on the expansion formula of the Pfaffian of skew-symmetric matrices~\cite{Ishikawa1995} and the characterization of extreme points in fractional matroid parity polytopes~\cite{Chang2001b}.
We remark that the formal statement is more subtle because of the absence of characterization of non-extreme half-integral points in the polytope; indeed we do not give a complete correspondence between $\rank X_i$ and $y_i$.
Nevertheless, our weak correspondence suffices to develop an algebraic algorithm that finds an optimal solution via the search-to-decision reduction by focusing on the lexicographically minimum optimal solution.

\paragraph{Faster algebraic algorithm with sparse representation.}
The above simple fractional linear matroid parity algorithm runs in $O(n^\omega + mn^2)$ time, where $2 \le \omega \le 3$ is the matrix multiplication exponent.
This is already faster than the known algorithm~\cite{Chang2001a} that takes $O(\CLVtime)$ time\footnote{To be fair, we remark that the algorithm~\cite{Chang2001a} can work \modify{for general matroids} in a more abstract setting. They assume a certain oracle that finds a basis of the intersection of a given line and flat. In the linear case, this oracle can be implemented to run in $O(n^\omega)$ time by Gaussian elimination. Chang et al.~\cite{Chang2001a} showed that their algorithm makes $O(m^4)$ oracle calls in total. Hence the total running time is $O(\CLVtime)$.}.
We further develop a much faster algorithm using the \emph{sparse matrix representation}, which is another matrix representation of the linear matroid parity problem introduced by Geelen and Iwata~\cite{Geelen2005} that extends a sparse representation of the linear matroid intersection problem~\cite{Harvey2009,Murota2000}; see also~\cite{Matoya2022}.
The current fastest algorithm for linear matroid intersection~\cite{Harvey2009} and parity~\cite{Cheung2014} both perform the search-to-decision reduction with the sparse representations and the divide-and-conquer strategy.
Adapting the same approach, we design a divide-and-conquer algorithm for the fractional linear matroid parity problem using the second-order blow-up of the sparse representation.
The resulting algorithm runs in $O(mn^{\omega-1})$ time.

\paragraph{Finding \modify{a} dominant 2-cover.}
The known algorithm for fractional linear matroid parity~\cite{Chang2001a} also finds the dual optimal solution called the \emph{dominant 2-cover}.
Both a maximum fractional matroid matching and the dominant 2-cover are necessary for the \emph{weighted} fractional matroid parity algorithm of Gijswijt and Pap~\cite{Gijswijt2013}.
Using our matrix representation and dual correspondence, we devise a simple randomized algorithm for finding the dominant 2-cover of fractional linear matroid matching in $O(mn^{\omega+1})$ time.
Combined with the above algebraic algorithm, this leads to an $O(mn^{\omega+4})$-time randomized algorithm for the weighted problem, which improves \modify{the} $O(\GPtime)$ time of Gijswijt and Pap~\cite{Gijswijt2013}\footnote{
    \modify{The algorithm of~\cite{Gijswijt2013} works for general matroids.}
}.

\subsubsection{Polyhedral Aspects of Fractional Linear Matroid Parity}
Before concluding our contributions, we remark that the fractional matroid parity polytope is interesting in its own right.
The fractional matroid parity polytope is defined by \modify{exponentially or} infinitely many linear constraints each of which arises from a subspace of $\K^n$.
This complicated nature of the linear system of the fractional matroid parity polytope makes it difficult to use the ellipsoid algorithm.
Indeed, the only known polynomial-time separation algorithm on the fractional matroid parity polytope is using the weighted algorithm of Gijswijt and Pap~\cite{Gijswijt2013} and the equivalence of separation and optimization.

In the special case that $\K$ is the complex field, the fractional matroid parity polytope coincides with the \emph{rank-two Brascamp-Lieb polytope}~\cite{Bennett2008,Valdimarsson2010}.
Garg et al.~\cite{Garg2018} devised a pseudo-polynomial time (weak) separation oracle for general-rank Brascamp-Lieb polytopes with operator scaling.
However, this does not yield a polynomial time separation algorithm even in the rank-two case.

Another potential approach for separation is submodular minimization over the product of diamond lattices~\cite{Fujishige2022}.
The separation problem can be reduced to minimizing a submodular function defined on a modular lattice $\mathcal{L}_1 \times \dotsb \times \mathcal{L}_m$, where $\mathcal{L}_i$ is the lattice of vector subspaces of a line $\ell_i \subseteq \K^n$.
The algorithm of Fujishige et al.~\cite{Fujishige2022} takes time polynomial in $n$ and $\sum_{i=1}^m |\mathcal{L}_i|$.
The lattice $\mathcal{L}_i$, however, becomes exponentially large or even infinite when $\K$ is an infinite field, hence this algorithm is inapplicable.

\subsection{Related Work and Future Directions}

\paragraph{Algebraic algorithms in combinatorial optimization.}
As mentioned above, algebraic algorithms using linear matrix representations have been developed for matching, linear matroid intersection and parity, which have $\{0,1\}$-optimal solutions.
A notable exception is an algebraic algorithm for $b$-matching~\cite{Gabow2021}.
A $b$-matching is a nonnegative integral vector on the edges of a graph, rather than a $\{0,1\}$-vector.
As a matrix representation of $b$-matching, they used a block matrix each of whose blocks is a low-rank symbolic matrix.
The resulting matrix is not a linear matrix but a quadratic matrix.
The low-rank block structure enables us to encode a maximum $b$-matching with the (standard) rank of the matrix.
We remark that their matrix representation has a pseudo-polynomial size.
On the other hand, our matrix representation for fractional matroid matching can encode half-integral solutions with a polynomial-size matrix using nc-rank.
We expect that matrix representations using nc-rank will become a powerful new tool for optimization problems that exhibit half-integrality.

\paragraph{Algorithms for non-commutative Edmonds' problem.}
We here overview the deterministic polynomial-time algorithms~\cite{Garg2020,Ivanyos2018,Hamada2021} for non-commutative Edmonds' problem.
These fall into three categories.
The first one~\cite{Garg2020} is the \emph{operator scaling} algorithm which extends the well-known Sinkhorn iteration and works when $\K$ is the rational field $\Q$.
The second one~\cite{Ivanyos2018} is based on the concept of the \emph{Wong sequence}, which is a linear algebraic generalization of augmenting paths in bipartite matching and linear matroid intersection.
The last one~\cite{Hamada2021} solves the dual problem of nc-rank as an optimization over a \emph{CAT(0) space}.

The operator scaling and Wong sequence algorithms have been originally proposed to compute the rank of linear matrices such that $\rank A = \ncrank A$ holds~\cite{Gurvits2004,Ivanyos2015}.
Tailored algorithms have been designed for some special classes of linear matrices with this property.
For example, the rank and the nc-rank of linear matrices $A$ such that $\rank A_i = 1$ for all $i$ are the same~\cite{Lovasz1989}.
Thus, any linear matroid intersection algorithm can be regarded as an algorithm for commutative and non-commutative Edmonds' problems with rank-one coefficients.
A generic \emph{partitioned matrix} with $2 \times 2$ submatrices provides another class of such property~\cite{Iwata1995}, for which a combinatorial algorithm was developed in~\cite{Hirai2021}.

\paragraph{Matrix representations for weighted problems.}
Weighted combinatorial optimization problems are represented by matrices in form of $A = \sum_{i=1}^m \theta^{w_i} x_i A_i$, where $\theta$ is a new indeterminate and $w_i$ is the weight of the $i$th element.
In this representation, the minimum weight of a solution can be retrieved from the degree (in $\theta$) of the determinant of $A$.
Faster algorithms for the weighed problems have been developed based on this representation~\cite{Sankowski2009,Harvey2009,Iwata2022}.

This matrix representation $A$ is a special class of \emph{linear polynomial matrices} $\sum_{k=0}^{\modify{d}} \theta^{\modify{k}} B_k$, where $B_0, \dotsc, B_d$ are $n \times n$ linear matrices with symbols $x_1, \dotsc, x_m$.
As a weighted analogue of non-commutative Edmonds' problem, Hirai~\cite{Hirai2019} introduced \emph{weighted non-commutative Edmonds' problem}, which is to compute the degree of the \emph{Dieudonné determinant} of a linear polynomial matrix.
Here, the Dieudonné determinant is a non-commutative generalization of the usual determinant.
This weighted variant can be deterministically solved in time polynomial in $n$, $m$, and $d$~\cite{Hirai2019}.
This result was later simplified by~\cite{Oki2023}.
A special class of linear polynomial matrices in form of $\sum_{i=1}^m \theta^{w_i} x_i A_i$, which \modify{is called \emph{a linear symbolic monomial matrix} in~\cite{Hirai2024}} and is important as matrix representations of weighted combinatorial optimization problems, can be solved in \emph{strongly} polynomial time~\cite{Hirai2022}; that is, the running time does not depend on $w_i$.
\modify{This algorithm is based on cost-scaling and simultaneous Diophantine approximation.}

\modify{
Extending our results, a recent paper~\cite{Hirai2024} establishes a matrix representation for weighted fractional linear matroid parity using the Dieudonné determinant.
The paper also gives min-max theorems and a strongly polynomial-time algorithm for computing the maximum degree of the Dieudonné determinant of a submatrix in a linear symbolic monomial matrix, extending the Hungarian method and the algorithm of Gijswijt--Pap~\cite{Gijswijt2013}.
}

\subsection{Organization}
The rest of this paper is organized as follows.
\Cref{sec:preliminaries} introduces preliminaries on the fractional linear matroid parity problem and the nc-rank of linear matrices.
Then, \Cref{sec:ncrank-fractional} shows that the nc-rank of the matroid representation of linear matroid parity corresponds to fractional linear matroid parity.
Based on this correspondence, \Cref{sec:algebraic-algorithm} develops an algebraic algorithm for the \modify{fractional} linear matroid parity problem and then \Cref{sec:finding-dual} presents an algorithm for the minimum 2-cover problem.
Finally, \Cref{sec:faster-algorithm} develops a faster divide-and-conquer algorithm with the sparse representation.

\section{Preliminaries}\label{sec:preliminaries}
We give basic definitions and notations.
Let $\N$ be the set of natural numbers and $\R$ the set of reals.
For $n \in \N$, let $[n] \coloneqq \{1, 2, \dotsc, n\}$.
For two real vectors $y = (y_1, \dotsc, y_m)$ and $z = (z_1, \dotsc, z_m)$, $y \le z$ means that $y_i \le z_i$ for all $i \in [m]$.
The \emph{cardinality} of a nonnegative vector $y \in \R^m$ is $|y| \coloneqq \sum_{i=1}^m y_i$.
Let $\bfzero$ and $\ones$ denote the all-zero and all-one vectors, respectively, of appropriate dimensions.
Let $e_i$ be the $i$th standard unit vector, i.e., its $j$th component is $1$ if $i = j$ and $0$ otherwise.

Let $\K$ be a ground field.
We assume that arithmetic operations on $\K$ can be performed in constant time except in \Cref{sec:bit-complexity} where the bit complexity is discussed for $\K = \Q$.
We denote by $\GL_n(\K)$ the set of $n \times n$ nonsingular matrices over $\K$.
For a matrix $A$, a row subset $I$, and a column subset $J$, we denote by $A[I, J]$ the submatrix indexed by $I$ and $J$, and by $A[I]$ the principal submatrix $A[I, I]$ for square $A$.
When $I$ (resp.\ $J$) is all the rows (resp.\ columns), we denote $A[I, J]$ by $A[*, J]$ (resp.\ $A[I, *]$).
The $n \times n$ identity matrix and the $n \times m$ zero matrix are denoted as $I_n$ and $O_{n, m}$, respectively.
We will omit the subscript of a zero matrix when its size does not matter.

A square matrix $A \in \K^{n \times n}$ is said to be \emph{skew-symmetric} if $A^\top = -A$ and its diagonals are zero; the latter condition \modify{concerns} the case where the characteristic of $\K$ is $2$.
For two vectors $a, b \in \K^n$, we define the \emph{wedge product} as $a \wedge b \coloneqq ab^\top - ba^\top$.
This is a skew-symmetric matrix of rank-two if $a$ and $b$ are linearly independent.
Conversely, any rank-two skew-symmetric matrix is the wedge product of some $a$ and $b$.

For $V, W \subseteq \K^n$, we mean by $V \le W$ that $V$ and $W$ are vector subspaces of $\K^n$ such that $V \subseteq W$, i.e., $V$ is a subspace of $W$, and $V < W$ means $V \le W$ and $V \ne W$.
For vectors $a_1, \dotsc, a_m \in \K^n$, let $\agbr{a_1, \dotsc, a_m}$ denote the vector subspace spanned by $a_1, \dotsc, a_m$.

\subsection{Linear Matroid Parity}\label{sec:linear-matroid-parity}
Let $\ell_1, \dotsc, \ell_m \le \K^n$ be two-dimensional vector subspaces, called \emph{lines}.
A line subset $M \subseteq L \coloneqq \{\ell_1, \dotsc, \ell_m\}$ is called a \emph{matroid matching} if it spans a $2|M|$-dimensional vector subspace of $\K^n$.
A \emph{parity base} is a matroid matching $M$ with $2|M| = n$.
Without loss of generality, we assume $n \le 2m$ since we can focus on the at most $2m$-dimensional subspace spanned by the lines.
The \emph{linear matroid parity problem}~\cite{Lawler1976} (or the \emph{linear matroid matching problem}) is to find a matroid matching of maximum cardinality.
This classical problem commonly generalizes the matching and linear matroid intersection problems and admits a min-max formula and polynomial-time algorithms~\cite{Lovasz1980a,Lovasz1980b,Lovasz1981,Gabow1986,Orlin1990,Orlin2008,Cheung2014}.

For the linear matroid parity problem, Lovász~\cite{Lovasz1979} introduced the following matrix representation:
\begin{align}\label{eq:parity-matrix}
    A = \sum_{i=1}^m x_i (a_i \wedge b_i),
\end{align}
where $\{a_i, b_i\}$ is any basis of $\ell_i$ for $i \in [m]$ and $x_1, \dotsc, x_m$ are distinct indeterminates.
That is, $A$ is a linear matrix~\eqref{def:linear} with rank-two skew-symmetric coefficients.
Recall that $\rank A$ denotes the rank of $A$ as a matrix over $\K(x_1, \dotsc, x_m)$.

\begin{theorem}[{\cite{Lovasz1979,Lovasz1989}}]\label{thm:parity-main}
    Let $A$ be the matrix representation~\eqref{eq:parity-matrix} corresponding to lines $L$.
    Then, we have
    \begin{align}
        \rank A = 2 \max_{\substack{M \subseteq L:\\ \text{matroid matching}}} |M|.
    \end{align}
\end{theorem}

Geelen and Iwata~\cite{Geelen2005} gave another matrix representation, called the \emph{sparse representation}.
Letting $\Delta = \begin{bNiceArray}{rr} 0 & +1 \\ -1 & 0 \end{bNiceArray}$ and $B_i = \begin{bmatrix} a_i & b_i \end{bmatrix}$ for $i \in [m]$, we define
\begin{align}\label{def:sparse}
    Z = \begin{bNiceArray}[margin]{c|ccc}
        O         & B_1        & \Cdots & B_m \\\midrule
        -B_1^\top & x_1 \Delta &        &     \\
        \Vdots    &            & \Ddots &     \\
        -B_m^\top &            &        & x_m \Delta
    \end{bNiceArray}.
\end{align}

\begin{theorem}[{Geelen and Iwata~\cite[Theorem~4.1]{Geelen2005}}]\label{thm:parity-main-sparse}
    Let $Z$ be the sparse representation~\eqref{def:sparse} corresponding to lines $L$ with $|L| = m$.
    Then, we have
    \begin{align}
        \rank Z = 2 \max_{\substack{M \subseteq L:\\ \text{matroid matching}}} |M| + 2m.
    \end{align}
\end{theorem}

\Cref{thm:parity-main-sparse} is obtained from \Cref{thm:parity-main} as follows.
Since $a_i \wedge b_i = B_i \Delta B_i^\top$, the Schur complement of the bottom-right block in $Z$ is $\sum_{i=1}^m x_i^{-1}(a_i \wedge b_i)$.
Thus, its rank is the same as $\rank A$.
By the Guttman rank additivity formula, we obtain $\rank Z = 2m + \rank A$.
See~\cite{Matoya2022}.

\subsection{Fractional Linear Matroid Parity}\label{sec:fractional-linear-matroid-parity}
The \emph{matroid parity polytope} is the convex hull of the incidence vectors of the matroid matchings.
In contrast to the matching and matroid intersection polytopes, a polyhedral description of matroid parity polytopes is still unknown.
As a relaxation of the matroid parity polytope, Vande Vate~\cite{VandeVate1992} introduced a \emph{fractional matroid parity} (\emph{matching}) \emph{polytope} as follows.
Let $L = \{\ell_1, \dotsc, \ell_m\}$ be lines.
A \emph{fractional matroid matching} is a nonnegative vector $y \in \R^m$ such that
\begin{align}\label{eq:fractional}
    \sum_{i=1}^m \dim (S \cap \ell_i) y_i \le \dim S
\end{align}
holds for all $S \le \K^n$.
The \emph{fractional matroid parity polytope} $P$ is the set of all fractional matroid matchings.
This polytope is half-integral, i.e., extreme fractional matroid matchings are half-integral, and the integral ones are the incidence vectors of the matroid matchings~\cite{VandeVate1992}.
We call $y \in P$ a \emph{fractional parity base} if $2|y| = n$.

The \emph{fractional linear matroid parity} (\emph{matching}) \emph{problem} is to find a fractional matroid matching of maximum cardinality.
Since fractional matroid parity polytopes are half-integral, there always exists a half-integral optimal solution.
Chang et al.~\cite{Chang2001a,Chang2001b} gave a min-max theorem and a polynomial-time algorithm for this problem.
To introduce the min-max theorem, we shall define a \emph{2-cover} as a pair $(S, T)$ of vector subspaces of $\K^n$ such that $\dim(S \cap \ell_i) + \dim(T \cap \ell_i) \geq 2$ for all $i \in [m]$.
A 2-cover $(S,T)$ is said to be \emph{nested} if $S \le T$.

\begin{theorem}[{\cite[Corollary~4.3]{Chang2001a}}]\label{thm:CLV}
    For a fractional matroid parity polytope $P$, it holds
    \begin{align}
        2\max_{y \in P} |y| = \min_{(S,T): \text{nested 2-cover}} (\dim S + \dim T).
    \end{align}
\end{theorem}

If $(S,T)$ and $(S',T')$ are minimum nested 2-covers, then $(S\cap S', T+T')$ and $(S + S', T \cap T')$ are also minimum 2-covers and the former is nested.
Hence, by the modularity of the dimension, there exists a unique minimum nested 2-cover $(S^*, T^*)$ such that $S^* \le S$ and $T \le T^*$ for any minimum nested 2-cover $(S, T)$~\cite[Lemma~4.9]{Chang2001b}.
This nested 2-cover $(S^*, T^*)$ is called the \emph{dominant 2-cover}.
The dominant 2-cover plays an important role in the weighted fractional matroid parity algorithm by Gijswijt and Pap~\cite{Gijswijt2013}.

Chang et al.~\cite{Chang2001b} characterized extreme fractional matroid matchings by means of \emph{canonical families}.
We introduce some definitions to describe it.
Let $\mathcal{S} = \{S_1, \dotsc, S_t\}$ be a chain of vector subspaces, i.e., $S_1 < S_2 < \dotsb < S_t \le \K^n$, and $L' \subseteq L$ be a line subset with $|L'| = t$.
Assuming $\ell \le S_t$ for all $\ell \in L'$, we define a graph $G(\mathcal{S}, L')$ in the following way.
The vertex and edge sets are identified with $\mathcal{S}$ and $L'$, respectively.
Every edge $\ell \in L'$ is incident to vertex $S_j$ if $\dim (S_j \cap \ell) - \dim (S_{j-1} \cap \ell) = 1$ and is a loop at $S_j$ if $\dim (S_j \cap \ell) - \dim (S_{j-1} \cap \ell) = 2$, where $S_0$ is regarded as $\{\bfzero\}$.
Let $y \in \zho^m$ be a half-integral vector.
We say that $S \le \K^n$ is \emph{tight} with respect to $y$ if the inequality~\eqref{eq:fractional} is satisfied with equality.
For $a \in \zho$, let $\supp_a(y) \coloneqq \{i \in [m]: y_i = a\}$ and $L_a(y) \coloneqq \{\ell_i: i \in \supp_a(y)\}$.
A chain $\mathcal{S} = \{S_1, \dotsc, S_t\}$ of length $t = \big|L_{\frac12}(y)\big|$ is called a \emph{canonical family} with respect to $y$ if the following are satisfied.
\begin{enumerate}
    \item $S_1, \dotsc, S_t$ are tight with respect to $y$.
    \item \modify{There exist $v_1, \dotsc, v_t \in \K^n$ such that $S_j = S_{j-1} + \agbr{v_j}$ for $j \in [t]$ with $S_0 = \sum_{\ell \in L_1(y)} \ell$.}
    \item The graph $G\bigl(\mathcal{S}, L_{\frac12}(y)\bigr)$ is the disjoint sum of node-disjoint odd cycles.
\end{enumerate}

\begin{theorem}[{\cite[Theorem~7.2]{Chang2001b}}]
    Let $P$ be the fractional matroid parity polytope defined by lines $L = \{\ell_1, \dotsc, \ell_m\}$.
    A half-integral vector $y \in \zho^m$ is an extreme point of $P$ if and only if $L_1(y)$ is a matroid matching and there is a canonical family with respect to $y$.
\end{theorem}

\subsection{Non-commutative Rank}
Let $A = \sum_{i=1}^m x_i A_i$ be a linear matrix~\eqref{def:linear} with $A_1, \dotsc, A_m \in \K^{n \times n}$.
As described in \Cref{sec:introduction}, the non-commutative rank (nc-rank) of $A$, denoted as $\ncrank A$, is equal to the following.
See~\cite{Cohn1985,Cohn1995,Fortin2004,Derksen2017} for details.
\begin{itemize}
    \item The \emph{inner rank} as a matrix over the free ring $R \coloneqq \K\agbr{x_1, \dotsc, x_m}$ generated by pairwise non-commutative indeterminates $x_1, \dotsc, x_m$.
          That is, the minimum $r$ such that $A$ is decomposed as $A = BC^\top$ with $B, C \in R^{n \times r}$.
    \item The rank as a matrix over the \emph{free skew field} $\F \coloneqq \K\fsbr{x_1, \dotsc, x_m}$ generated by pairwise non-commutative indeterminates $x_1, \dotsc, x_m$.
          That is, the maximum size of a nonsingular (invertible) submatrix of $A$ over $\F$.
          The free skew field is a quotient of $\K\agbr{x_1, \dotsc, x_m}$ defined by Amitsur~\cite{Amitsur1966}.
    \item $\frac{1}{d} \rank \blowup{A}{d}$ for $d \ge n-1$, where $\blowup{A}{d}$ is the $d$th-order blow-up~\eqref{def:blow-up} of $A$.
\end{itemize}

In general, the rank and nc-rank of a linear matrix $A$ satisfy $\rank A \le \ncrank A \le 2 \rank A$~\cite[Corollary~2]{Fortin2004}.
Generalizing the König--Egeváry theorem for bipartite matching and Edmonds' matroid intersection theorem, Fortin and Reutenauer~\cite{Fortin2004} presented the following min-max type formulation for nc-rank.

\begin{theorem}[{\cite[Theorem~1]{Fortin2004}}]\label{thm:FR}
    For an $n \times n$ linear matrix $A$, it holds that
    \begin{align}\label{eq:FR}
    \ncrank A = \min\left\{2n-s-t : P, Q \in \GL_n(\K),  PAQ = \begin{bmatrix}
        * & * \\
        O_{s,t} & *
    \end{bmatrix}\right\}.
    \end{align}
\end{theorem}

Hamada and Hirai~\cite{Hamada2021} rephrased \Cref{thm:FR} as follows.

\begin{theorem}[{\cite{Hamada2021}}]\label{thm:MVSP}
    For an $n \times n$ linear matrix $A$, it holds that
    \begin{align}\label{eq:MVSP}
        \ncrank A = \min\left\{
            2n-\dim X-\dim Y
            : \text{$X, Y \le \K^n$, $A_i(X, Y) = \{0\}$ for $i \in [m]$}
        \right\},
    \end{align}
    where $A_i(X, Y) \coloneqq \{x^\top A_i y: x \in X, y \in Y\}$.
\end{theorem}

The dual problem~\eqref{eq:MVSP} is called the \emph{minimum vanishing subspace problem} (MVSP).
It is known that the MVSP is submodular function minimization on the product of the lattice of all vector subspaces of $\K^n$ and its order-reversed lattice~\modify{\cite{Hamada2021,Ivanyos2022}}.
Namely, if $(X, Y)$ and $(X', Y')$ attain the minimum, so do $(X + X', Y \cap Y')$ and $(X \cap X', Y + Y')$.

\subsection{Linear Algebra Toolbox}
We collect useful tools in linear algebra.
First, we deal with the Kronecker product.
Recall that the Kronecker product of an $n \times m$ matrix $A = (a_{ij})$ and a $p \times q$ matrix $B$ is an $np \times mq$ matrix
\begin{align}
    A \otimes B \coloneqq \begin{bNiceMatrix}
        a_{11} B & \Cdots & a_{1m} B \\
        \Vdots & \Ddots & \Vdots \\
        a_{n1} B & \Cdots & a_{nm} B
    \end{bNiceMatrix}.
\end{align}
For matrices $A, B, C$, and $D$ of such size that $AB$ and $CD$ are defined, the Kronecker product satisfies the \emph{mixed-product property}
\begin{align}\label{eq:mixed-product}
    (AB) \otimes (CD) = (A \otimes C)(B \otimes D).
\end{align}
Applying~\eqref{eq:mixed-product} twice, for matrices $A, B, C, D, E, F$ of appropriate size, we obtain
\begin{align}\label{eq:mixed-product-three}
    (ABC) \otimes (DEF) = (A \otimes D)((BC) \otimes (EF)) = (A \otimes D)(B \otimes E)(C \otimes F).
\end{align}

Next, let $A = (a_{ij})$ be an $n \times n$ skew-symmetric matrix with $n$ being even.
The \emph{Pfaffian} of $A$ is
\begin{align}
    \pf A \coloneqq \sum_{\sigma \in F_n} \sgn \sigma \prod_{i \in [n]: even} a_{\sigma(i-1)\sigma(i)},
\end{align}
where $F_n$ is the set of permutations $\sigma: [n] \to [n]$ such that $\sigma(1) < \sigma(3) < \dotsb < \sigma(\modify{n-1})$ and $\sigma(i-1) < \sigma(i)$ for even $i \in [n]$.
For convenience, let $\pf A = 0$ when $n$ is odd.
The Pfaffian satisfies ${(\pf A)}^2 = \det A$, meaning that $A$ is nonsingular if and only if $\pf A \ne 0$.

We give two expansion formulas of Pfaffian.
The first one is easily obtained from the definition of Pfaffian.

\begin{proposition}[{see~\cite[Lemma~7.3.20]{Murota2000}}]\label{prop:pf-expansion}
    For a skew-symmetric matrix $A = Q + T$, we have
    \begin{align}
        \pf A = \sum_{S} \pm \pf Q[S] \pf T[\overline{S}],
    \end{align}
    where $S$ runs over all row (column) subsets of $A$ and $\overline{S}$ is the complement of $S$.
\end{proposition}

The second one is a generalization of the Cauchy--Binet formula.

\begin{proposition}[{\cite{Ishikawa1995}}]\label{prop:ishikawa-wakayama}
    For skew-symmetric $A \in K^{n \times n}$ and $B \in K^{m \times n}$, it holds
    \begin{align}
        \pf BAB^\top = \sum_{J} \det B[*, J] \pf A[J],
    \end{align}
    where $J$ runs over all row (column) subsets of $A$ of size $m$.
\end{proposition}

We next describe results on block matrices.

\begin{lemma}[see {\cite[Section 4.3]{Harvey2009}}]
    Let $T$ be a nonsingular matrix and
    \begin{align}
        Z = \begin{bmatrix}
            O & Q_1 \\
            Q_2 & T
        \end{bmatrix}.
    \end{align}
    Let $M = -Q_1 T^{-1}Q_2$ be the Schur complement of $T$.
    If $M$ is nonsingular, $Z$ is also nonsingular and
    \begin{align}\label{eq:Schur-inverse}
        Z^{-1} = \begin{bmatrix}
            M^{-1} & -M^{-1}Q_1T^{-1} \\
            -T^{-1}Q_2 M & T^{-1} + T^{-1}Q_2M^{-1}Q_1T^{-1}
        \end{bmatrix}.
    \end{align}
\end{lemma}

\begin{lemma}[{\cite[Corollary~2.2]{Harvey2009}}]\label{lem:small-update}
Let $Z$ be a nonsingular matrix and $Z'$ be a matrix identical to $Z$ \modify{except} that $Z[S] \neq Z'[S]$ for some subset $S$.
Let $M = Z^{-1}$ and $D = Z - Z' \neq O$.
Then, $Z'$ is nonsingular if and only if $I_{|S|} + D[S] M[S]$ is nonsingular.
If $Z'$ is nonsingular, then
\begin{align}
    {(Z')}^{-1} = M - M[*, S]{\bigl(I_{|S|} + D[S] M[S]\bigr)}^{-1} D[S] M[S, *].
\end{align}
\end{lemma}

By~\Cref{lem:small-update}, given $M[S]$, $M[S, S']$, $M[S', S]$, and $M[S']$, we can compute ${(Z')}^{-1}[S']$ in $O(|S|^\omega + |S||S'|^{\omega - 1})$ time for any subset $S'$.

\section{Non-commutative Rank and Fractional Linear Matroid Parity}\label{sec:ncrank-fractional}
In this section, we show the following non-commutative and fractional counterpart to \Cref{thm:parity-main}.

\begin{theorem}\label{thm:main}
    Let $P$ be a fractional matroid parity polytope and $A$ the corresponding matrix representation~\eqref{eq:parity-matrix}.
    Then, we have
    \begin{align}
        \ncrank A = 2\max_{y \in P} |y|.
    \end{align}
\end{theorem}

The proof of \Cref{thm:main} is given in \Cref{sec:proof-main}.
In \Cref{sec:tutte-matrices}, we consider the special case of \Cref{thm:main} for the matching problem.

\subsection{Proof of \texorpdfstring{\Cref{thm:main}}{Theorem \ref{thm:main}}}\label{sec:proof-main}

First, for a skew-symmetric linear matrix, we show that one can attain the same objective value as~\eqref{eq:FR} with \emph{simultaneous} row and column operations.

\begin{lemma}\label{lem:skew-symmetric-FR}
    For an $n \times n$ skew-symmetric linear matrix $A$, we have
    \begin{align}\label{eq:skew-symmetric-FR}
        \ncrank A =
        \min\left\{
            2n - s - t :
            P \in \GL_n(\K),
            PAP^\top =
            \begin{bNiceMatrix}[first-row,first-col]
                    & n-s & s-t & t  \\
                n-s & * & * & *  \\
                s-t & * & * & O  \\
                t   & * & O & O  \\
            \end{bNiceMatrix}
         \right\}.
    \end{align}
\end{lemma}
\begin{proof}
    Let $(X, Y)$ be a minimizer of the MVSP~\eqref{eq:MVSP}.
    Then $(Y, X)$ is also a minimizer because $A$ is skew-symmetric.
    By the submodularity of the MVSP, so are $(X + Y, X \cap Y)$ and $(X \cap Y, X + Y)$.
    Letting $s = \dim (X + Y)$ and $t = \dim (X \cap Y)$, we can construct a nonsingular matrix $P$ so that its bottom $s$ rows and bottom $t$ rows span $X + Y$ and $X \cap Y$, respectively.
    By $A_i(X + Y, X \cap Y) = A_i(X \cap Y, X + Y) = \{0\}$ for $i \in [m]$, the matrix $PAP^\top$ has the desired structure of zero blocks.
\end{proof}

We now prove \Cref{thm:main} with the aid of \Cref{lem:skew-symmetric-FR}.

\begin{proof}[{Proof of \Cref{thm:main}}]
    First, we show that $\ncrank A \leq 2\max_{y \in P} |y|$.
    Let $B_i = \begin{bmatrix} a_i & b_i \end{bmatrix}$ for $i \in [m]$ and $(S, T)$ be a minimum nested 2-cover.
    By appropriate change of basis, we can assume that $S = \langle e_1, \dots, e_s \rangle$ and $T = \langle e_1, \dots, e_t \rangle$ for $s \leq t$.
    Since $(S,T)$ is a 2-cover, the column-echelon form of $B_i$ must have one of the following block structures:
    \begin{align}
        \begin{bNiceMatrix}[first-col]
            s   & * & * \\
            t-s & * & \bfzero \\
            n-t & * & \bfzero
        \end{bNiceMatrix},
        \begin{bmatrix}
            * & * \\
            * & * \\
            \bfzero & \bfzero
        \end{bmatrix},
        \begin{bmatrix}
            * & * \\
            * & \bfzero \\
            \bfzero & \bfzero
        \end{bmatrix},
        \begin{bmatrix}
            * & * \\
            \bfzero & \bfzero \\
            \bfzero & \bfzero
        \end{bmatrix}.
    \end{align}
    Note that column operations do not change wedge product $a_i \wedge b_i$.

    We will show that the same transformation yields a common $(n-t) \times (n-s)$ block for each $a_i \wedge b_i$.
    It suffices to consider the first two cases because the others yield a no smaller zero block.
    If $B_i \sim
        \begin{bmatrix}
            * & * \\
            * & \bfzero \\
            * & \bfzero
        \end{bmatrix}$, then
    \[
        a_i \wedge b_i =
        \begin{bNiceMatrix}[first-row, first-col]
                & s & t-s & n-t \\
        s    & * & *   & *   \\
        t-s  & * & O   & O   \\
        n-t  & * & O   & O
        \end{bNiceMatrix}.
    \]
    If $B_i \sim
        \begin{bmatrix}
            * & * \\
            * & * \\
            \bfzero & \bfzero
        \end{bmatrix}$, then
    \[
        a_i \wedge b_i =
        \begin{bNiceMatrix}[first-row, first-col]
                & s & t-s & n-t \\
        s    & * & *   & O   \\
        t-s  & * & *   & O   \\
        n-t  & O & O   & O
        \end{bNiceMatrix}.
    \]
    Therefore, the right bottom $(n-t) \times (n-s)$ zero block is common for all $a_i \wedge b_i$.
    By \Cref{thm:CLV,thm:FR}, we have $\ncrank A \leq n - (n-t) - (n-s) = s+t = 2\max_{y \in P} |y|$.

    We show the other direction.
    Let $P \in \GL_n(\K)$ be an optimal solution in \Cref{lem:skew-symmetric-FR} and $s, t$ ($s \ge t$) be the values in~\eqref{eq:skew-symmetric-FR} for $P$.
    Let $\tilde B_i = P B_i$.
    For $p, q \in [n]$ ($p \neq q$), let us denote by $\det \tilde B_i[p,q]$ the $2 \times 2$ minor corresponding to the $p$th and $q$th rows of $\tilde B_i$.
    Then, every minor $\det \tilde B_i[p, q]$ vanishes for $p > n-s$ and $q > n - t$, because it equals the $(p,q)$-entry of $P (a_i \wedge b_i) P^\top$.
    This implies the row echelon-form of $\tilde B_i$ must be one of the following block structures:
    \[
        \begin{bNiceMatrix}[first-col]
            n-s   & * & * \\
            s-t & * & \bfzero \\
            t & * & \bfzero
        \end{bNiceMatrix},
        \begin{bmatrix}
            * & * \\
            * & * \\
            \bfzero & \bfzero
        \end{bmatrix},
        \begin{bmatrix}
            * & * \\
            * & \bfzero \\
            \bfzero & \bfzero
        \end{bmatrix},
        \begin{bmatrix}
            * & * \\
            \bfzero & \bfzero \\
            \bfzero & \bfzero
        \end{bmatrix}.
    \]
    Therefore, letting $S, T$ be the subspaces spanned by the first $(n-s)$ and $(n-t)$ columns of $P$, respectively, $(S,T)$ is a nested 2-cover with $\dim S + \dim T = 2n -s -t$.
    This completes the proof.
\end{proof}

\modify{In the above proof, we have established the following, which we record for later use.
\begin{lemma}\label{lem:2-cover-vs-MVSP}
    If $(S,T)$ is a minimum nested 2-cover with $\dim S = s$ and $\dim T = t$, then $(\overline{S}, \overline{T})$ is a maximum vanishing subspace with $\dim \overline{S} = n - s$ and $\dim \overline{T} = n - t$, where $\overline{S}$ and $\overline{T}$ denote the direct complements of $S$ and $T$, respectively.
    Conversely, if $(X, Y)$ is a maximum vanishing subspace with $\dim X = s$ and $\dim Y = t$, then $(\overline X, \overline Y)$ is a minimum 2-cover with $\dim \overline X = n - s$ and $\dim \overline Y = n - t$.
\end{lemma}
}

The following result on the nc-rank of the sparse representation~\eqref{def:sparse} is obtained from \Cref{thm:main} in the same way as \Cref{thm:parity-main-sparse}.

\begin{corollary}
    Let $P$ be a fractional matroid parity polytope defined from lines $L$ with $|L| = m$ and $Z$ the corresponding sparse representation~\eqref{def:sparse}.
    Then, we have
    \begin{align}
        \ncrank Z = 2\max_{y \in P} |y| + 2m.
    \end{align}
\end{corollary}

\begin{remark}\label{rmk:one-d-line}\modify{
    Gijswijt--Pap~\cite{Gijswijt2013} dealt with a little more general setting of the fractional matroid parity problem, where lines are allowed to be one-dimensional subspaces.
    We can generalize \Cref{thm:main} to this setting by considering a rank-one symmetric matrix $a_i a_i^\top$ instead of $a_i \wedge b_i$ for a one-dimensional line $\ell_i = \agbr{a_i}$.
    While the resulting linear symbolic matrix $A$ is no longer skew-symmetric, \Cref{lem:skew-symmetric-FR} is still valid for this $A$.
    Then, considering $(S, T)$ with $\dim(S \cap \ell_i) + \dim(T \cap \ell_i) \ge \dim \ell_i$ instead of 2-covers, we can show that $\ncrank A$ is equal to $\max_{y \in P} \sum_{i=1}^m \dim \ell_i \cdot y_i$ in the same manner as the proof of \Cref{thm:main}.
}\end{remark}

\subsection{Nc-rank of Tutte Matrices}\label{sec:tutte-matrices}
Let $G$ be a \modify{loopless} graph with vertices $V(G) = \{v_1, \dotsc, v_n\}$.
The \emph{Tutte matrix} $T_G$ of $G$ is an $n \times n$ linear matrix given as
\begin{align}
    T_G = \sum_{e = \{v_i, v_j\} \in E(G), i < j} x_e (e_i \wedge e_j),
\end{align}
where $x_e$ is an indeterminate indexed by $e \in E(G)$.
The coefficient matrix $e_i \wedge e_j$ is the rank-two skew-symmetric matrix whose $(i, j)$-entry is $+1$, $(j, i)$-entry is $-1$, and others are $0$.
The rank of $T_G$ is twice the size of a maximum matching in $G$~\cite{Tutte1947}.

Now \Cref{thm:main} reveals what value of $G$ the nc-rank of $T_G$ represents.
First, $T_G$ coincides with the matrix representation of linear matroid parity corresponding to the matching problem on $G$.
In addition, when $G$ has no loop, the fractional version of this instance of linear matroid parity is nothing but the \emph{fractional matching problem} on $G$~\cite{VandeVate1992}.
Here, a \emph{fractional matching} of $G$ is a nonnegative vector $y \in \R^m$ ($m = |E(G)|$) such that $\sum_{e \in \delta(v)} y_e \le 1$ for all $v \in V(G)$, where $\delta(v)$ is the set of edges incident to $v \in V(G)$.
Thus, as a corollary of \Cref{thm:main}, we obtain:

\begin{corollary}
    The nc-rank of the Tutte matrix of a loopless graph is twice the maximum cardinality of a fractional matching.
\end{corollary}

\modify{We can also deal with graphs with self-loops as the way in \cref{rmk:one-d-line}, because self-loops correspond to one-dimensional lines.}

\section{Algebraic Algorithm}\label{sec:algebraic-algorithm}
In this section, we present an algebraic algorithm for the fractional linear matroid parity problem that outputs not only the optimal value but also an optimal solution by combining \Cref{thm:main} and the search-to-decision reduction.

\subsection{Algorithm Description}
Let $P$ be the fractional matroid parity polytope defin\modify{e}d from lines $\ell_i = \agbr{a_i, b_i}$ ($i \in [m]$) and $A$ be the corresponding matrix representation~\eqref{eq:parity-matrix}.
For a half-integral vector $y \in \zho^m$, \modify{we define a $2n \times 2n$ matrix}
\begin{align}
    \blowup{A}{2}(y) \coloneqq \sum_{i=1}^m Y_i \otimes (a_i \wedge b_i),
\end{align}
where $Y_i = U_i U_i^\top$ and $U_i$ is a $2 \times 2y_i$ matrix with indeterminates in its entries for $i \in [m]$.
Namely, $\blowup{A}{2}(y)$ is obtained by substituting the $2 \times 2$ symmetric matrix $Y_i$ of rank $2y_i$ into $X_i$ in the second-order blow-up $\blowup{A}{2}$ for $i \in [m]$.
Note that $\blowup{A}{2}(y)$ is skew-symmetric as each $Y_i$ is symmetric\modify{, and each entry in $\blowup{A}{2}(y)$ is a homogeneous quadratic polynomial in entries of $U_1, \dotsc, U_m$}.
We let $\rho_A(y) \coloneqq \rank \blowup{A}{2}(y)$.

\begin{algorithm}[tb]
    \caption{Simple algebraic algorithm for the fractional linear matroid parity problem}\label{alg:simple}
    \begin{algorithmic}[1]
        \Input{A fractional matroid parity polytope $P$ given as $a_i, b_i\in \K^n$ ($i \in [m]$) and a finite subset $R$ of $\K$.}
        \Output{The lexicographically minimum maximum fractional matroid matching in $P$}
        \State{Let $A = \sum_{i=1}^m x_i (a_i \wedge b_i)$.}
        \State{Estimate $\rho_A(\cdot)$ by substituting elements of $R$ uniformly at random in the following.}
        \State{$r \coloneqq \rho_A(\ones)$}{\label{lst:simple-alg-init}}
        \State{$y \gets \ones$}
        \For{$i = 1, \dotsc, m$}
            \If{$\rho_A\left(y - \frac12 e_i\right) = r$}{\label{lst:simple-alg-if1}}
                \State{$y_i \gets \frac12$}{\label{lst:simple-alg-if-inner-begin}}
                \If{$\rho_A(y - \frac12 e_i) = r$}{\label{lst:simple-alg-if2}}
                    \State{$y_i \gets 0$}{\label{lst:simple-alg-if-inner-end}}
                \EndIf
            \EndIf
        \EndFor
        \State{\Return{$y$}}
    \end{algorithmic}
\end{algorithm}

\Cref{alg:simple} describes the presented algebraic algorithm.
The algorithm iteratively computes $\rho_A(y)$ for different $y \in \zho^m$, which can be efficiently performed via the random substitution from a finite subset $R \subseteq \K$.
We will show that the value $r = \rho_A(\ones)$ computed in \Cref{lst:simple-alg-init} is four times the cardinality of maximum fractional matroid matching, and $\rho_A(y)$ computed in \Cref{lst:simple-alg-if1,lst:simple-alg-if2} are equal to $r$ if and only if there exists $z \in \zho^m$ such that $z \le y$, $z \in P$, and $|z| = \frac{r}{4}$.
Thus, \Cref{alg:simple} finds a maximum fractional matroid matching via the search-to-decision reduction.
Specifically, the algorithm outputs the \emph{lexicographically minimum} optimal solution.
In the rest of this section, we give proof of these facts and then show the following conclusion.

\begin{theorem}\label{thm:simple-algorithm-is-valid}
    If $|R| \geq 16mn$, \Cref{alg:simple} finds the lexicographically minimum vector among all maximum fractional matroid matchings in $P$ in $O(n^\omega+mn^2)$ time with probability at least $\frac{1}{2}$.
\end{theorem}

\subsection{Characterizing Rank of Second-order Blow-up}
By the skew-symmetricity, $\rank \blowup{A}{2}(y)$ is equal to the maximum size of a nonsingular principal submatrix.
We first give an expansion formula of the Pfaffian of nonsingular principal submatrices of $\blowup{A}{2}(y)$ and use it to characterize the rank of $\blowup{A}{2}(y)$.
Let $B_i = \begin{bmatrix} a_i & b_i \end{bmatrix}$ for $i \in [m]$.

\begin{lemma}\label{lem:pf-expand}
    For $y \in \zho^m$ and $I \subseteq [2n]$, it holds
    \begin{align}\label{eq:pf-expand}
        \pf \blowup{A}{2}(y)[I]
        =
        \sum_{\substack{z \in \zho^m: \\ |z| = \frac{|I|}{4},\, z \le y}}
        \sum_{(J_1, \dotsc, J_m) \in \mathcal{J}^y(z)}
        \tau_{J_1, \dotsc, J_m},
    \end{align}
    where $\mathcal{J}^y(z)$ is the family of $m$-tuples $(J_1, \dotsc, J_m)$ such that
    \begin{gather}
        J_i = \begin{cases}
            \{1, 2, 3, 4\} & (z_i = 1), \\
            \text{$\{1, 2\}$ or $\{3, 4\}$} & \left(y_i = 1, z_i = \frac12\right), \\
            \{1, 2\} & \left(y_i = z_i = \frac12\right), \\
            \emptyset & (z_i = 0).
        \end{cases}
    \shortintertext{and}
        \tau_{J_1, \dotsc, J_m} = \det \begin{bNiceMatrix} (U_1 \otimes B_1)[I, J_1] & \Cdots & (U_m \otimes B_m)[I, J_m] \end{bNiceMatrix}.\label{eq:c_z-summand}
    \end{gather}
\end{lemma}

\begin{proof}
    Recall from \Cref{sec:linear-matroid-parity} that the wedge product $a_i \wedge b_i$ is written as $B_i \Delta B_i^\top$ with $\Delta = \begin{bNiceArray}{rr} 0 & +1 \\ -1 & 0 \end{bNiceArray}$.
    Using the mixed-product property~\eqref{eq:mixed-product-three}, we obtain
    \begin{gather}
        Y_i \otimes (a_i \wedge b_i)
        = \bigl(U_i I_{2y_i} {U_i}^\top\bigr) \otimes \bigl(B_i \Delta B_i^\top\bigr)
        = (U_i \otimes B_i) (I_{2y_i} \otimes \Delta) {(U_i \otimes B_i)}^\top
    \shortintertext{and}
        \blowup{A}{2}(y) = B(y) D(y) {B(y)}^\top,
    \shortintertext{where}
        B(y) = \begin{bNiceMatrix} U_1 \otimes B_1 & \Cdots & U_m \otimes B_m \end{bNiceMatrix},
        \quad
        D(y) = \begin{bNiceMatrix}
            I_{2y_1} \otimes \Delta &&\\
            & \Ddots & \\
            && I_{2y_m} \otimes \Delta
        \end{bNiceMatrix}.\label{def:B-y}
    \end{gather}
    Thus, we have $\blowup{A}{2}(y)[I] = B(y)[I, *] D(y) {B(y)[I, *]}^\top$.
    Applying \Cref{prop:ishikawa-wakayama}, we have
    \begin{align}\label{eq:pf-principal-submatrix}
        \pf \blowup{A}{2}(y)[I]
        = \pf B(y)[I, *] D(y) {B(y)[I, *]}^\top
        = \sum_{J} \det B(y)[I, J] \pf D(y)[J],
    \end{align}
    where $J$ runs over all column subsets in $B(y)$ of cardinality $|I|$.
    Letting $J_i$ be the columns of $(U_i \otimes B_i)[I, *]$ in $B(y)[I, J]$, we have
    \begin{align}
        \pf \blowup{A}{2}(y)[I] = \sum_{\substack{(J_1, \dotsc, J_m):\\\sum_{i=1}^m |J_i| = |I|}} \tau_{J_1, \dotsc, J_m} \prod_{i=1}^m \pf (I_{2y_i} \otimes \Delta)[J_i].
    \end{align}
    Let $z_i = \frac{|J_i|}{4}$ for $i \in [m]$.
    Now $\pf (I_{2y_i} \otimes \Delta)[J_i] \in \{0, 1\}$ and it does not vanish if and only if $|J_i| \le 4y_i$ and $J_i = \{1, 2, 3, 4\}$, $\{1, 2\}$, $\{3, 4\}$ (allowed only when $y_i = 1$), or $\emptyset$.
    Thus, $z$ corresponding to non-vanishing terms is a half-integral vector with $z \le y$.
    The equation~\eqref{eq:pf-principal-submatrix} is represented as~\eqref{eq:pf-expand} in this way.
\end{proof}

An important consequence of \Cref{lem:pf-expand} is the following.

\begin{lemma}\label{lem:characterize-nonsingular-principal-submatrix}
    For $y \in \zho^m$, $\rho_A(y)$ is equal to four times the maximum cardinality of $z \in \zho^m$ such that $z \le y$ and $B(z)$ is of column-full rank, where $B(z)$ is defined by~\eqref{def:B-y}.
\end{lemma}

\begin{proof}
    Since $\blowup{A}{2}(y)$ is skew-symmetric, $\rho_A(y) = \rank \blowup{A}{2}(y)$ is equal to the maximum cardinality of $I \subseteq [2n]$ such that $\blowup{A}{\modify{2}}(Y)[I]$ is nonsingular.
    Fix $I \subseteq [2n]$.
    By \Cref{lem:pf-expand}, $\pf \blowup{A}{2}(Y)[I]$ is expanded as~\eqref{eq:pf-expand}.
    Now the summand~\eqref{eq:c_z-summand} of~\eqref{eq:pf-expand} for $z$ is the same polynomial as $\det B(z)[I]$ up to the indeterminates' labeling.
    Thus, if $B(z)[I]$ is singular for any $z \in \zho^m$ with $z \le y$ and $|z| = \frac{|I|}{4}$, the principal submatrix $\blowup{A}{2}(y)[I]$ must be singular.
    Conversely, if there exists $z \in \zho^m$ such that $z \le y$, $|z| = \frac{|I|}{4}$, and $B(z)[I]$ is nonsingular, the principal submatrix $\blowup{A}{2}(y)[I]$ becomes nonsingular because different $z$ and $(J_1, \dotsc, J_m) \in \mathcal{J}^y(z)$ yield summands~\eqref{eq:c_z-summand} that do not cancel out.

    To summarize, $\blowup{A}{2}(y)[I]$ is nonsingular if and only if there exists $z \in \zho^m$ with $z \le y$ such that $B(z)[I]$ is nonsingular.
    Finding a maximum $I$ satisfying these conditions, we obtain the claim.
\end{proof}

For later use, we provide a different expansion formula of $\pf\blowup{A}{2}(y)$.

\begin{lemma}\label{lem:naive-expansion}
    For $y \in \zho^m$ and $I \subseteq [2n]$, it holds
    \begin{align}
        \pf \blowup{A}{2}(y)[I] = \sum_{\substack{z \in \zho^m:\\ |z| = \frac{|I|}{4},\, z \le y}}
        \sum_{(S_1, \dots, S_m) \in \mathcal{S}(z)} \lambda_{S_1, \dotsc, S_m} \prod_{i=1}^m \pf(Y_i \otimes \Delta)[S_i],
    \end{align}
    where $\mathcal{S}(z)$ is the family of $m$-tuples $(S_1, \dots, S_m)$ such that
    \begin{gather}
        S_i =
        \begin{cases}
            \{1,2,3,4\} & (z_i = 1), \\
            \{1,2\}, \{1,4\}, \{2,3\}, \text{or } \{3,4\} & \bigl(z_i = \frac12\bigr), \\
            \emptyset & (z_i = 0).
        \end{cases}\label{eq:non-vanishing-sets}
    \shortintertext{and}
        \lambda_{S_1, \dotsc, S_m} = \det \begin{bNiceMatrix}
            (I_2 \otimes B_1)[I, S_1] & \Cdots & (I_2 \otimes B_m)[I, S_m]
        \end{bNiceMatrix}\label{def:lambda}
    \end{gather}
\end{lemma}
\begin{proof}
    Using a different factorization
    \begin{align}\label{eq:naive-factor}
        Y_i \otimes (a_i \wedge b_i)
        = \bigl(I_{2} {Y_i} I_{2}\bigr) \otimes \bigl(B_i \Delta B_i^\top\bigr)
        = (I_2 \otimes B_i) (Y_i \otimes \Delta) {(I_2 \otimes B_i)}^\top,
    \end{align}
    we obtain $\blowup{A}{2}(y) = B E(y) B^\top$, where
    \begin{align}
        B = \begin{bNiceMatrix}
            I_2 \otimes B_1 & \Cdots & I_2 \otimes B_m
        \end{bNiceMatrix},
        \quad
        E(y) = \begin{bNiceMatrix}
            Y_1 \otimes \Delta & & \\
            & \Ddots & \\
            & & Y_m \otimes \Delta
        \end{bNiceMatrix}.
    \end{align}
    Applying \Cref{prop:ishikawa-wakayama}, we obtain
    \begin{align}
        \blowup{A}{2}(y)[I] = \sum_{(S_1, \dots, S_m): \sum_{i=1}^m |S_i| = |I|} \lambda_{S_1, \dotsc, S_m} \prod_{i=1}^m \pf(Y_i \otimes \Delta)[S_i].
    \end{align}
    Let $z_i = \frac{|S_i|}{4}$ for each $i$.
    The term corresponding to $(S_1, \dots, S_m)$ vanishes if any of $S_i$ is odd because $Y_i \otimes \Delta$ is skew-symmetric.
    Hence, $z$ is half-integral for non-vanishing terms.
    If $z_i > y_i$, which implies $|S_i| > 4y_i$, then $\pf (Y_i \otimes \Delta)[S_i] = 0$ because the rank of $Y_i \otimes \Delta$ is $4y_i$.
    So $z \leq y$.
    Finally, it is easy to check that $\pf (Y_i \otimes \Delta)[S_i] = 0$ if $S_i$ does not satisfy~\eqref{eq:non-vanishing-sets}.
\end{proof}

\subsection{Full-column Rankness and Half-integral Fractional Matroid Matchings}

\Cref{lem:characterize-nonsingular-principal-submatrix} characterizes the $\rho_A(y)$ value in terms of the \modify{full-}column rankness of $B(y)$.
We next relate $B(y)$ and half-integral points in $P$.

\begin{lemma}\label{lem:characterize-half-integral-matching}
    A half-integral vector $y \in \zho^m$ is in $P$ if $B(y)$ is of full-column rank.
\end{lemma}

\begin{proof}
    \modify{Suppose to the contrary that there is a subspace $S \le \K^n$ such that $\sum_{i=1}^m d_i y_i > s$, where $d_i = \dim (S \cap \ell_i)$ and $s = \dim S$.
    By appropriate change of basis, we can assume $S = \agbr{e_1, \dotsc, e_s}$.
    Similarly, by changing a basis of $\ell_i$, we can assume that $S \cap \ell_i$ is spanned by the left $d_i$ columns of $B_i$ for $i \in [m]$.
    Then, $B_i$ has the form}
    \begin{align}\modify{
        B_i = \begin{bNiceMatrix}[first-row, first-col]
             & d_i & 2 - d_i \\
        s    & * & * \\
        n-s  & O & *
        \end{bNiceMatrix}.
    }\end{align}
    \modify{Consider a matrix $B' = \begin{bNiceMatrix} U_1 \otimes B_1[*, [d_1]] & \Cdots & U_m \otimes B_m[*, [d_m]] \end{bNiceMatrix}$, which is a submatrix of $B(y)$ obtained by extracting columns of $B(y)$ corresponding to left $d_i$ columns of $B_i$ for each $i \in [m]$.
    Then, $B'$ has $2s$ nonzero rows and $2 \sum_{i=1}^m d_i y_i$ columns, meaning that $B'$ is not of full-column rank.
    Hence, $B(y)$ is not of full-column rank as well.}
\end{proof}

The converse of \Cref{lem:characterize-half-integral-matching} holds for extreme points in $P$.

\begin{lemma}\label{lem:extreme-converse}
    The matrix $B(y)$ with $y \in \zho^m$ is of full-column rank if $y$ is an extreme point of $P$.
\end{lemma}

\begin{proof}
    Let $s = \big|L_1(y)\big|$, $t = \big|L_{\frac12}(y)\big|$, $S_0 = \sum_{\ell \in L_1(y)} \ell$, and $\mathcal{S} = \{S_1, \dotsc, S_t\}$ ($S_1 < \dotsb < S_t$) be a canonical family with respect to $y$ (see \Cref{sec:fractional-linear-matroid-parity}).
    By change of basis, we can assume that $S_j = \langle e_1, \dotsc, e_{2s+j} \rangle$ for $j = 0, \dotsc, t$.
    Each edge $\ell_i \in L_{\frac12}(y)$ in the graph $G\bigl(\mathcal{S}, L_{\frac12}(y)\bigr)$ has ends $\{S_{j_i}, S_{j'_i}\}$ with $j_i, j'_i \in [t]$ satisfying
    \begin{align}
        \dim (S_{j_i} \cap \ell_i) &= \dim (S_{j_i-1} \cap \ell_i) + 1, \\
        \dim (S_{j'_i} \cap \ell_i) &= \dim (S_{j'_i-1} \cap \ell_i) + 1.
    \end{align}
    These equalities imply that $j_i \ne j'_i$, i.e., $G\bigl(\mathcal{S}, L_{\frac12}(y)\bigr)$ has no self-loop.
    In addition, by appropriately taking a basis $\{a_i, b_i\}$ of $\ell_i$, we can transform the matrix $B_i$ as the form of
    \[
        B_i = \begin{bNiceMatrix}[last-col=3]
            * & * & \\
            1 & * & \text{$j_i$th} \\
            0 & * & \\
            0 & 1 & \text{$j'_i$th} \\
            0 & 0 &
        \end{bNiceMatrix}.
    \]
    Conversely, for each $j \in [t]$, there exists exactly two edges $\ell_{i_j}, \ell_{i'_j}$ incident to the node $S_j$.
    Note that $i_j$ and $i'_j$ are distinct as $G\bigl(\mathcal{S}, L_{\frac12}(y)\bigr)$ has no self-loop.

    \modify{
        Based on the above arguments, we permute the columns of $B(y)$ as follows.
        The $4s$ columns corresponding to the lines in $L_1(y)$ are placed at the leftmost.
        For each $j \in [t]$, the $(4s + 2j - 1)$st and $(4s + 2j)$th columns consist of one column each of $U_{i_j} \otimes B_{i_j}$ and $U_{i'_j} \otimes B_{i'_j}$.
        This column permutation yields the following block upper-triangular form:
    }
    \begin{align}
        \begin{bNiceMatrix}
            T & *       & *        & *       & *        & \Cdots & *       & *        \\
              & U_{i_1} & U_{i'_1} & *       & *        & \Cdots & *       & *        \\
              &         &          & U_{i_2} & U_{i'_2} & \Ddots & \Vdots  & \Vdots   \\
              &         &          &         &          & \Ddots & *       & *        \\
              &         &          &         &          &        & U_{i_t} & U_{i'_t} \\
              &         &          &         &          &        &         &
        \end{bNiceMatrix},
    \end{align}
    where $T$ is a matrix of order $\modify{4}s$.
    \modify{The matrix $T$ is nonsingular because the lines in $L_1(y)$ are linearly independent.}
    \modify{In addition,} each diagonal block $\begin{bmatrix} U_{i_j} & U_{i'_j} \end{bmatrix}$ is nonsingular \modify{since $U_{i_j}$ and $U_{i'_j}$ are matrices of different indeterminates}.
    \modify{Therefore}, $B(y)$ is of full-column rank.
\end{proof}

It is unknown whether $B(y)$ is of full-column rank even for a half-integral but non-extreme $y \in P$.
Nevertheless, our weak characterization is enough to show the validity of \Cref{alg:simple}.

\subsection{Proof of \texorpdfstring{\Cref{thm:simple-algorithm-is-valid}}{Theorem \ref{thm:simple-algorithm-is-valid}}}

\Cref{lem:characterize-half-integral-matching,lem:extreme-converse,lem:characterize-nonsingular-principal-submatrix} are aggregated into the following lemma.

\begin{lemma}\label{lem:blowup-rank}
    For $y \in \zho^m$, it holds that
    \begin{align}\label{eq:blowup-rank}
        \modify{4\max \{ |z| : z \in P_{\mathrm{ext}}, z \le y \} \le {}}
        \rho_A(y) \le 4 \max \{ |z| : z \in P, z \le y \},
    \end{align}
    where $P_\mathrm{ext}$ is the set of extreme points of $P$.
\end{lemma}

\begin{proof}
    \modify{For an extremet point $z \in P_\mathrm{ext}$, the matrix $B(z)$ is of full column rank by \Cref{lem:extreme-converse}, meaning $4|z| \le \rho_A(y)$ by \Cref{lem:characterize-nonsingular-principal-submatrix}. Thus, the first inequality holds.}
    Next, by \Cref{lem:characterize-nonsingular-principal-submatrix} again, $\rho_A(y)$ is equal to four times the maximum cardinality of $z \in \zho^m$ such that $z \le y$ and $B(z)$ is of column-full rank.
    Since such $z$ is in $P$ by \Cref{lem:characterize-half-integral-matching}, \modify{the second inequality also holds}.
\end{proof}

\modify{\Cref{lem:blowup-rank} implies that $\rho_A(y) = 4 \max \{ |z| : z \in P, z \le y \}$ holds if the maximum $z$ is attained by an extreme point of $P$.}
Combining \Cref{thm:main} and \Cref{lem:blowup-rank} with $y = \ones$, we obtain the following theorem, which states that the second-order blow-up is sufficient for $A$ to attain the nc-rank.

\begin{theorem}\label{thm:blowup-size-2}
    Let $P$ be a fractional matroid parity polytope and $A$ the corresponding matrix representation~\eqref{eq:parity-matrix}.
    Then, we have
    \begin{align}\label{eq:blowup-size-2}
        \max_{y \in P} |y| = \frac{1}{4} \rank \blowup{A}{2}.
    \end{align}
\end{theorem}

\begin{proof}\modify{
    We have $4 \max_{y \in P} |y| = \rho_A(\ones)$ from \Cref{lem:blowup-rank} and the fact that a linear function is maximized at an extreme point of a polytope, meaning~\eqref{eq:blowup-size-2}.}
\end{proof}

Now we are ready to prove \Cref{thm:simple-algorithm-is-valid}.

\begin{proof}[{Proof of \Cref{thm:simple-algorithm-is-valid}}]
    First, we show the validity of the algorithm assuming that the exact value of $\rho_A$ can be computed.
    Let $y^*$ be the lexicographically minimum point among all maximum points in $P$.
    Note that $y^*$ is half-integral since $y^*$ is extreme.
    Let $y^{(0)} = \ones$ and $y^{(i)}$ denote $y$ in \Cref{alg:simple} at the end of the $i$th iteration for $i \in [m]$.
    We show by induction on $i$ the following claim: for every $i = 0, \dotsc, m$, it holds that $y^{(i)}_j = y^*_j$ if $j \le i$ and $y^{(i)}_j = 1$ if $j > i$.
    Then, \Cref{thm:simple-algorithm-is-valid} is obtained as the case for $i = m$.

    The claim for $i = 0$ is trivial.
    Suppose that the claim is true for $i - 1$ and consider the case for $i$.
    Let $y$ be the candidate solution given to $\rho_A$ in \Cref{lst:simple-alg-if1} of \Cref{alg:simple}.
    Suppose the case when $y^*_i = 1$.
    Then $y$ is lexicographically smaller than $y^*$ by the inductive assumption.
    This means that that there is no maximum point $z \in P$ satisfying $z \le y$.
    By \Cref{lem:blowup-rank}, we have
    \[
        \rho_A(y)
        \le 4 \max\{ |z| : z \in P, z \le y \}
        < 4 \max\{ |z| : z \in P \}
        = \rho_A(\ones)
    \]
    and thus the $i$th iteration in \Cref{alg:simple} does not execute Lines~\ref{lst:simple-alg-if-inner-begin}--\ref{lst:simple-alg-if-inner-end} and $y_i^{(i)}$ is fixed to $1$.
    Suppose the case when $y^*_i \le \frac{1}{2}$.
    Then $y^* \le y$ by the inductive assumption.
    We have
    \[
        \max \{|z| : z \in P\}
        = |y^*|
        \le \max\{|z| : z \in P, z \le y\}
        \le \max \{|z| : z \in P\},
    \]
    which means that $y^*$ attains the maximum in $\max\{|z| : z \in P, z \le y\}$.
    Since $y^*$ is extreme, by \Cref{lem:blowup-rank}, we obtain
    \[
        \rho_A(y) = 4 \max\{|z| : z \in P, z \le y\} = 4|y^*| = \rho_A(\ones).
    \]
    Thus the $i$th iteration goes to \Cref{lst:simple-alg-if-inner-begin}.
    The same argument can be applied to the conditional branch in \Cref{lst:simple-alg-if2}.
    This completes the proof of validity assuming that all the evaluations of $\rho_A$ is exact.

    Next, we analyze the probability of success when we estimate $\rho_A(y)$ by substituting uniform random elements from $R$ to the entries of $U_i$.
    Since each entry of $\blowup{A}{2}(y) = \sum_{i=1}^m U_i U_i^\top \otimes (a_i \wedge b_i)$ is quadratic, the degree of any non-vanishing $k \times k$ minor is $2k$ for any $k \leq 2n$.
    Therefore, the probability that such a non-vanishing minor remains non-vanishing after the random substitution from $R$ is at least $1 - \frac{4n}{|R|}$ by the Schwartz-Zippel lemma.
    Since there are at most $2m$ evaluations of $\rho_A$, it suffices to take $|R| = 16mn$ to guarantee that all the evaluations of $\rho_A$ during the algorithm are correct with probability at least $\frac12$.
    Thus, the output of the algorithm is correct with probability at least $\frac12$.

    Finally, we analyze the time complexity.
    Each evaluation of $\rho_A$ can be done in $O(n^2)$ time using low-rank updates.
    Let $y$ be a tentative solution, $y' = y - \frac12 e_i$ be a candidate solution in the algorithm, and $U_i$ be a random matrix of size $2 \times 2y_i$ for $i \in [m]$.
    Suppose that we already have a rank-revealing decomposition $PQ$ of $\blowup{A}{2}(y)$.
    Then, to evaluate $\rho_A(y')$, it suffices to compute a rank-revealing decomposition of $M' = PQ + D_i \otimes (a_i \wedge b_i)$, where
    \begin{align}
        D_i = \begin{cases}
            -U_i[*, \{1\}] {U_i[*, \{1\}]}^\top & \bigl(y_i = 1, y_i' = \frac12\bigr), \\
            -U_i U_i^\top & \bigl(y_i = \frac12, y_i' = 0\bigr).
        \end{cases}
    \end{align}
    Note that the rank of $D_i$ is one, so that of $D_i \otimes (a_i \wedge b_i)$ is two.
    Hence, we can use a low-rank update formula for a rank-revealing decomposition to compute the rank of $M'$ in $O(n^2)$ time.
    For the first evaluation of $\rho_A$ in \Cref{lst:simple-alg-init}, we simply compute a rank-revealing decomposition in $O(n^\omega)$ time.
    Thus, the time complexity of Algorithm~\ref{alg:simple} is $O(n^\omega + mn^2)$.
\end{proof}

\section{Finding Dominant 2-Cover}\label{sec:finding-dual}

In this section, we devise a randomized algorithm for finding the dominant 2-cover.
Our algorithm is based on the concept of the Wong sequence introduced by Ivanyos et al.~\cite{Ivanyos2015}.
In this section, we regard a linear matrix $A = \sum_{i=1}^m x_i A_i$ as a matrix space $\caA$ spanned by $A_i$, where $A_i = a_i \wedge b_i$ ($i=1,\dots,m$).
The nc-rank of $\caA$ (denoted by $\ncrank\caA$) is defined to be that of $A$.
For a subspace $U \leq \K^n$, we define $\caA(U) \coloneqq \agbr{Mu : M \in \caA, u \in U}$.
A subspace $U \leq \K^n$ is called a \emph{$c$-shrunk subspace} if $\dim U - \dim\caA(U) \geq c$.
In the following, we only consider $c = n - \ncrank \caA$\modify{, i.e., we only consider maximally shrunk subspaces}. 

\modify{There is a simple correspondence between $c$-shrunk subspaces, maximum vanishing subspaces, and minumum 2-covers, which follows immediately from \Cref{lem:2-cover-vs-MVSP}.
\begin{lemma}\label{lem:shrunk-vs-2-cover}
If $U$ is a $c$-shrunk subspace, then $(\overline{\caA(U)}, U)$ is a maximum vanishing subspace, and $(\caA(U), \overline{U})$ is a minimum 2-cover, where $\overline{\caA(U)}$ denotes the direct complement of $\caA(U)$.
Especially, if $U$ is the $c$-shrunk subspace with minimum dimension, then $(\caA(U), \overline{U})$ is the dominant 2-cover.
\end{lemma}
}

For $M \in \caA$, the (second) Wong sequence of $(M, \caA)$ is defined as
\begin{alignat}{3}
    W_0 &:= \{\bfzero\},  & \quad W_i &:= \caA(M^{-1}(W_{i-1}))  \quad (i=1, 2, \dots),
\end{alignat}
where $M^{-1}(W_{i-1})$ is the preimage of $W_{i-1}$ under $M$.
The following lemma states the necessary properties of the Wong sequence.

\begin{lemma}[{\cite{Ivanyos2015}}]\label{lem:Wong}
    For the Wong sequence $(W_i)$ of $(M, \caA)$, the following holds.
   \begin{enumerate}
       \item $W_i$ is nonincreasing, i.e., $W_{i-1} \leq W_i$ for $i=1, 2, \dots$. (Hence there exists the limit of $W_i$, which we denote by $W_\infty$)
       \item $\rank M = \ncrank \caA$ if and only if $W_\infty \leq \im M$.
       \item If $M \in \caA$ satisfies $\rank M = \ncrank \caA$, then $W_i = {(\caA M^+)}^i\ker \modify{M}$ for $i=0,1,2,\dots$, where $M^+$ is a pseudoinverse\footnote{A pseudoinverse of $M$ is a nonsingular matrix $M^+$ such that the restriction of $MM^+$ on $\im A$ is the identity map. A pseudoinverse can be found by Gaussian elimination.} of $M$.
       \item If $M \in \caA$ satisfies $\rank M = \ncrank \caA$, then $U^* := M^{-1}(W_\infty)$ is the minimum $c$-shrunk subspace.
   \end{enumerate}
\end{lemma}

We now describe our algorithm.
The minimum blow-up size of $\caA$ is two by \Cref{thm:blowup-size-2}, so $\blowup{\caA}{2}$ contains a matrix $M$ such that $\rank M = 2\ncrank\caA$.
Such a maximum-rank matrix can be found with high probability by the Schwartz-Zippel lemma.
Then, we compute the limit of Wong sequence $W_\infty$ of $(M, \blowup{\caA}{2})$.
By \Cref{lem:Wong}, $U_* = M^{-1}W_\infty$ is the minimum $2c$-shrunk subspace of $\blowup{\caA}{2}$.
It is easy to compute the minimum $c$-shrunk subspace $U_0^*$ of $\caA$ from $U_*$ as shown below.
Once we obtain the minimum $c$-shrunk subspace $U_0^*$ of $\caA$, the dominant 2-cover $(S^*, T^*)$ is obtained by the correspondence between shrunk-subspaces and 2-covers shown in the previous section.
A pseudocode of the algorithm is shown in \Cref{alg:2-cover}.

\begin{algorithm}
    \caption{Randomized algorithm to find the dominant 2-cover}\label{alg:2-cover}
    \begin{algorithmic}[1]
        \Input rank-two skew-symmetric $A_i$ ($i \in [m]$) and a finite subset $R$ of $\K$.
        \Output the dominant 2-cover $(S^*, T^*)$.
        \State Draw random $2 \times 2$ matrices $X_i$ ($i \in [m]$), where each entry of $X_i$ is  independently drawn from $R$ uniformly at random.
        \State Compute $M = \sum_{i=1}^m A_i \otimes X_i \in \blowup{\caA}{2}$.
        \State Compute the limit of the Wong sequence $W_\infty$ of $(M, \blowup{\caA}{2})$.
        \If{$W_\infty \not\leq \im M$}
        \State \textbf{fail}.
        \Else
        \State Compute bases of $U^* = M^{-1}(W_\infty)$ and $U^*_0 = \{ u \in \R^n : u \otimes e_i \in U^* \text{ for $i \in [2]$} \}$ by Gaussian elimination.
        \State Let $S^*$ be the direct complement of $U_0^*$ and $T^* = \caA(U_0^*)$.
        \State \Return $(S^*, T^*)$.
        \EndIf
    \end{algorithmic}
\end{algorithm}

We have the following theorem.

\begin{theorem}\label{thm:2-cover}
    If $|R| \geq 4n$, \Cref{alg:2-cover} finds the dominant 2-cover in $O(mn^{\omega+1})$ time with probability at least $\frac12$.
\end{theorem}

To prove the theorem, we need a few lemmas.
Let $\mathrm{M}_d(\K)$ denote the set of $d \times d$ matrices over $\K$.

\begin{lemma}[{cf.~\cite[Proposition~5.2]{Ivanyos2017}}]
    Let $d \in \N$.
    If a subspace $W \leq \K^{nd}$ satisfies $(I_n \otimes \mathrm{M}_d(\K)) W = W$, then
    $W = W_0 \otimes \K^d$, where $W_0 = \{w \in \K^n : u \otimes e_i \in W \text{ for all $i \in [d]$} \}$.
\end{lemma}
\begin{proof}
    To see $W_0 \otimes \K^d \leq W$, take $w = w_0 \otimes v$ for $w_0 \in W_0$ and $v \in \K^d$ arbitrarily.
    Then, $w = w_0 \otimes \bigl(\sum_{i=1}^d v_i e_i\bigr) = \sum_{i=1}^d v_i (w_0 \otimes e_i) \in W$.
    On the other hand, take an \modify{arbitrary} $w \in W$ and write $w = \sum_{j=1}^{d} w_j \otimes e_j \in W$, where $w_j \in \K^d$ is the $j$th subvector of $w$.
    Then, we have $(I_n \otimes E_{ij}) w = w_j \otimes e_i \in W$ for $i, j \in [d]$ by $(I_n \otimes \mathrm{M}_d(\K)) W = W$.
    Therefore, $w_j \in W_0$ for $j \in [d]$, which means $w \in W_0 \otimes \K^d$.
    Hence $W_0 \otimes \R^d \geq W$.
\end{proof}

\begin{lemma}\label{lem:shrunk-blowup}
    Let $U$ be a $2c$-shrunk subspace of $\blowup{\caA}{2}$. Then, $U = U_0 \otimes \K^2$ and $U_0$ is an $c$-shrunk subspace for $\caA$.
    If $U$ is minimum, then $U_0$ is also minimum.
\end{lemma}
\begin{proof}
    For notational simplicity, we denote $\blowup{\caA}{2}$ by $\caB$.
    First note that $\caB$ satisfies $(I_n \otimes \mathrm{M}_2(\K)) \caB = \caB$ and $\caB (I_n \otimes \mathrm{M}_2(\K)) = \caB$.
    Let $W = \caB U$. Then,
    \[
        (I_n \otimes \mathrm{M}_2(\K)) W = (I_n \otimes \mathrm{M}_2(\K)) \caB  U = \caB U = W.
    \]
    By the previous lemma, $W = W_0 \otimes \K^2$.
    Let $U' = (I \otimes \mathrm{M}_2(\K)) U$.
    Then, $U' \geq U$ and $\caB U' = \caB (I_n \otimes \mathrm{M}_2(\K)) U = \caB U = W$, so
    $
       2c \geq \dim U' - \dim \caB U' \geq \dim U - \dim \caB U = 2c.
    $
    Thus, $\dim U' = \dim U$ and hence $U' = (I_n \otimes \mathrm{M}_2(\K)) U = U$.
    By the previous lemma, $U = U_0 \otimes \K^2$.
    \modify{Then,
    \[
        W_0 \otimes \K^2 = W = \caB U = (\caA \otimes \mathrm{M}_2(\K))(U_0 \otimes \K^2) = (\caA U_0) \otimes \K^2,
    \]
    so we have $W_0 = \caA U_0$.
    }
    Thus,
    \[
     \dim U_0 - \dim \caA U_0 \modify{{}={}} \dim U_0 - \dim W_0 = \frac{\dim U - \dim W}{2} = c
     \]
     and $U_0$ is a $c$-shrunk subspace for $\caA$.

     If $U'$ is a $c$-shrunk subspace of $\caA$, then $U' \otimes \K^2$ is a $2c$-shrunk subspace of $\caA^{\{2\}}$.
     Therefore, $U' \otimes \K^2 \geq U = U_0 \otimes \K^2$ if $U$ is the minimum $2c$-shrunk subspace of $\caA^{\{2\}}$.
     Hence, $U' \leq U_0$.
\end{proof}

\begin{proof}[Proof of \Cref{thm:2-cover}]
    As $|R| \geq 4n$, the random matrix $M$ attains the maximum rank in $\blowup{\caA}{2}$ with probability at least $\frac12$ by the Schwartz-Zippel lemma.
    We can verify whether $M$ attains the maximum rank by checking $W_\infty \leq \im A$ by \Cref{lem:Wong}.
    Given the maximum-rank $M$, the computed subspace $U^*_0$ is the minimum $c$-shrunk subspace of $\caA$ by \Cref{lem:Wong,lem:shrunk-blowup}.
    So $(S^*, T^*)$ is the dominant 2-cover by \modify{\Cref{lem:shrunk-vs-2-cover}}.

    Now we analyze the time complexity.
    A basis of the limit of the Wong sequence is found in $O(mn^{\omega+1})$ time as follows.
    First, we compute a basis of $\ker M$ and a pseudoinverse $M^+$ of $M$ in $O(mn^{\omega-1})$ time by Gaussian elimination.
    This is done by factorization~\eqref{eq:naive-factor}.
    By \Cref{lem:Wong}, $W_i = {(\caB M^+)}^i \ker M$ for each $i$.
    So given a basis of $W_{i-1}$, we can obtain a basis of $W_i$ in $O(mn^\omega)$ time.
    By repeating this at most $2n$ times, we obtain a basis of the limit $W_\infty$ in $O(mn^{\omega+1})$ time.
    Once we obtain $W_\infty$, the remaining operations can be done in $O(mn^{\omega})$ time.
\end{proof}

\begin{remark}
\modify{Given an extreme maximum fractional matroid matching, one can find the dominant 2-cover with the unweighted algorithm by Chang~et~al.~\cite{Chang2001a}. Therefore, one can combine our algebraic algorithms with their algorithm in a black-box manner to find the dominant 2-cover. This yields an alternative randomized algorithm to find the dominant 2-cover. However, it is unclear whether Chang~et~al.'s algorithm has polynomial bit complexity over the rational field. On the other hand, the algorithm presented above does have polynomial bit complexity thanks to the connection to the Wong sequence, which is known to have polynomial bit complexity~\cite{Ivanyos2015}.}
\end{remark}

\section{Faster Algorithm with Sparse Representation}\label{sec:faster-algorithm}
Algorithm~\ref{alg:simple} finds a maximum fractional matroid matching in $O(n^\omega + mn^2)$ time.
In this section, we design a faster algorithm.
For the simplicity of the presentation, we first describe algorithms for finding a fractional parity base.
We explain how to extend them to maximum fractional matroid matching in \Cref{subsec:extension}.

A fractional matroid matching is said to be \emph{extendible} if there is a fractional parity base $z$ such that $y \le z$.
For $y \in \zho^m$, we define a matrix $Z(y)$ by
\begin{align}
    Z(y) =
    \begin{bNiceArray}[margin]{c|ccc}
        O                     & I_2 \otimes B_1 & \Cdots & I_2 \otimes B_m \\ \midrule
        -I_2 \otimes B_1^\top & Y_1 \otimes \Delta &  &  \\
         \Vdots               &  & \Ddots   & \\
        -I_2 \otimes B_m^\top &  &  & Y_m \otimes \Delta
    \end{bNiceArray},
\end{align}
where $Y_i = U_i U_i^\top$ and $U_i$ is a $2 \times 2(1-y_i)$ matrix with indeterminates.
When $y = \bfzero$, $Z(y)$ coincides with the sparse representation~\eqref{def:sparse}.
We show an analogue of \Cref{lem:characterize-nonsingular-principal-submatrix} for $Z(y)$.

\begin{lemma}\label{lem:extensible}
    For $y \in \zho^m$, the matrix $Z(y)$ is nonsingular if and only if there exists a half-integral $z \ge y$ such that $|z| = \frac{n}{2}$ and $B(z)$ is nonsingular.
\end{lemma}

\begin{proof}
    Without loss of generality, we assume that $\supp_1(y) = \{1, \dots, s\}$ and $\supp_{\frac12}(y) = \{s+1, \dots, t\}$ for some $s \leq t$.
    Let us write $Z(y) = Q + T$ where
    \begin{align}
        Q & =
    \begin{bNiceArray}[margin]{c|ccc}
        O                     & I_2 \otimes B_1 & \Cdots & I_2 \otimes B_m \\ \midrule
        -I_2 \otimes B_1^\top & \\
        \Vdots                & \\
        -I_2 \otimes B_m^\top &
    \end{bNiceArray}, \\
    T &=
    \begin{bNiceArray}[margin]{c|ccc}
        O & \\ \midrule
        \phantom{I_2 \otimes B_1^\top} & Y_1 \otimes \Delta &  &  \\
                                       &  & \Ddots   & \\
                                       &  &  & Y_m \otimes \Delta
    \end{bNiceArray}.
    \end{align}
    Then, by \Cref{prop:pf-expansion}, we have
    \begin{align}
        \pf Z(y) = \sum_{S} \pm \pf Q[\overline{S}] \pf T[S],
    \end{align}
    where $S$ runs over all column subsets of $Z(y)$.
    Consider $S$ such that $\pf Q[\overline{S}]$ and $\pf T[S]$ are both nonzero.
    First, $S \cap [n] = \emptyset$, otherwise $T[S]$ contains zero rows.
    For each $i \in [m]$, let $S_i$ denote the intersection of $S$ and the columns containing $I_2 \otimes B_i$.
    By the block-diagonal structure of $T$, $\pf T[S] = \prod_{i=1}^m \pf (Y_i \otimes \Delta)[S_i]$ and
    \begin{itemize}
        \item $S_i = \emptyset$ for $i \leq s$; otherwise $T[S]$ contains zero rows.
        \item For $s < i \leq t$,
        $S_i$ is either $\{1,2\}$, $\{1,4\}$, $\{2,3\}$, or $\{3,4\}$.
        \item For other $i$, $S_i$ is even.
    \end{itemize}
    Letting $z'$ be a vector such that $z'_i = \frac{|S_i|}{4}$, this condition is equivalent to $z'$ is half-integral and $z' \leq \ones-y$.
    On the other hand, we have
    \begin{align}
        \pf Q[\overline{S}]
        = \lambda_{\overline{S_1}, \dotsc, \overline{S_m}},
    \end{align}
    where $\overline{S_i}$ is the complement of $S_i$ within the columns containing $I_2 \otimes B_i$ and $\lambda_{\overline{S_1}, \dotsc, \overline{S_m}}$ is the polynomial defined by~\eqref{def:lambda} with $I = [2n]$.
    Thus, $\pf Q[\overline{S}] \neq 0$ implies $\sum_{i=1}^m |\overline{S_i}| = 2n$, or equivalently, $|S| = 4m - 2n$.
    Grouping the terms with $z' \in \zho^m$, we obtain
    \begin{align}
        \pf Z(y)
        &= \sum_{\substack{z' \in \zho^m:\\ |z'| = m - \frac{n}{2}, z' \le \ones-y}}
        \sum_{(S_1, \dots, S_m) \in \mathcal{S}(z')} \lambda_{\overline{S_1}, \dotsc, \overline{S_m}}
        \prod_{i=1}^m \pf(Y_i \otimes \Delta)[S_i] \\
        &= \sum_{\substack{z' \in \zho^m:\\ |z'| = m - \frac{n}{2}, z' \le \ones-y}}
        \sum_{(S_1, \dots, S_m) \in \mathcal{S}(\ones - z')} \lambda_{S_1, \dotsc, S_m}
        \prod_{i=1}^m \pf(Y_i \otimes \Delta)[\overline{S_i}],
    \end{align}
    where $\mathcal{S}(z')$ is defined in \Cref{lem:naive-expansion} and the second equality follows since $(S_1, \dots, S_m) \in \mathcal{S}(z')$ if and only if $(\overline{S_1}, \dots, \overline{S_m}) \in \mathcal{S}(\ones - z')$.
    Now observe that the RHS is a nonzero polynomial if and only if
    \begin{align}
        \sum_{\substack{z' \in \zho^m:\\ |z'| = m - \frac{n}{2}, z' \le \ones-y}}
        \sum_{(S_1, \dots, S_m) \in \mathcal{S}(\ones - z')} \lambda_{S_1, \dotsc, S_m}
        \prod_{i=1}^m \pf(Y_i \otimes \Delta)[S_i]
    \end{align}
    is a nonzero polynomial.
    By \Cref{lem:naive-expansion,lem:characterize-nonsingular-principal-submatrix}, the latter is nonzero if and only if there exists a half-integral $z' \leq \ones-y$ such that $|z'| = m - \frac{n}{2}$ and $B(\ones - z')$ is nonsingular.
    By changing a variable $z = \ones - z'$, we see that this is equivalent to the existence of a half-integral $z \geq y$ such that $|z| = \frac{n}{2}$ and $B(z)$ is nonsingular.
\end{proof}

\begin{algorithm}[tb]
    \caption{An $O(m^\omega)$-time algorithm to find a fractional parity base}\label{alg:sparse}
    \begin{algorithmic}[1]
        \Input{A fractional matroid parity polytope $P$ given as $a_i, b_i\in \K^n$ ($i \in [m]$) and a finite subset $R$ of $\K$.}
        \Output{The lexicographically maximum fractional parity base in $P$}
        \State{$y \gets \bfzero$ and $L \gets [m]$.}
        \State{Compute a matrix $Z(y)$ and substitute random elements from $R$ uniformly at random.}
        \State{Compute $M = {Z(y)}^{-1}$.}
        \State \Return $\textsc{BuildFractionalParityBase}(L, y, M)$
        \item[]
        \Function{BuildFractionalParityBase}{$L = \{a, \dots, b\}, y, M$}
        \State\Comment{Invariant: (i) $y$ is extensible and (ii) $M = Z(y)^{-1}[L]$.}
        \If{$|L| \geq 2$}
            \State Let $L_1 = \bigl\{a, \dots, \floor{\frac{a+b}{2}}\bigr\}$ and $L_2 = \bigl\{\floor{\frac{a+b}{2}} + 1, \dots, b\bigr\}$.
            \State $y \gets \textsc{BuildFractionalParityBase}(L_1, y, M[L_1])$ \label{line:BFMM-L1-update}
            \State Compute $M[L_2] := {Z(y)}^{-1}[L_2]$ in $O(|L|^\omega)$ time as shown in Lemma~\ref{lem:small-update}.
            \State $y \gets \textsc{BuildFractionalParityBase}(L_2, y, M[L_2])$
            \State \Return $y$
        \Else
            \State Take the unique line $a$ in $L$.
            \State Let $D_{\frac12} = Z\bigl(y + \frac{1}{2} e_a\bigr)- Z(y)$ and $D_{1} = Z(y + e_a)- Z(y)$.
            \State $y_a \gets \begin{cases}
                1   & (\det(I_{|L|} + D_{1}[L]M[L]) \neq 0), \\
                \frac{1}{2} & \bigl(\det\bigl(I_{|L|} + D_{\frac12}[L]M[L]\bigr) \neq 0, \det(I_{|L|} + D_{1}[L]M[L]) = 0\bigr), \\
                0   & \bigl(\det\bigl(I_{|L|} + D_{\frac12}[L]M[L]\bigr) = 0\bigr).
            \end{cases}$ \label{line:BFMM-x-update}
            \State \Return $y$
        \EndIf
        \EndFunction
    \end{algorithmic}
\end{algorithm}

For a set $L$ of lines, we abuse the notation $Z(y)[L]$ to denote the $4|L|\times 4|L|$ principal submatrix of $Z(y)$ corresponding to the lines in $L$.

\subsection{An \texorpdfstring{$O(m^\omega)$}{O(m\textasciicircum omega)}-Time Algorithm}
Now we describe an $O(m^\omega)$-time algorithm using $Z(y)$ under the assumption that $P$ has a fractional parity base.
The algorithm keeps a fractional matroid matching $y$ and a matrix $M$.
Initially, $y = \bfzero$ and $M = {Z(y)}^{-1}$; note that $Z(y)$ is nonsingular by \Cref{lem:extreme-converse,lem:extensible}.
The main algorithm calls subroutine \textsc{BuildFractionalParityBase}.
\textsc{BuildFractionalParityBase} takes a set $L$ of lines, a fractional matroid matching $y$, and a matrix $M$ as inputs, which satisfies the two invariants; (i) $y$ is extensible and (ii) $M = {Z(y)}^{-1}[L]$.
The base case is that $L$ consists of a single line, say $a$.
Then, it tries to increase the $a$th component of $y$, while keeping that $y$ is extensible.
If $L$ consists of more than one lines, then \textsc{BuildFractionalParityBase} takes an (almost) equipartition $L_1, L_2$ of $L$ and recurse on $L_1$.
If $y$ is updated on the recursion, we compute $M = {Z(y)}^{-1}[L_2]$ in $O(|L|^\omega)$ time by Lemma~\ref{lem:small-update} to keep the invariant.
Finally, we recurse on $L_2$ and return the output.
The presudocode is given in Algorithm~\ref{alg:sparse}.

Similar to Algorithm~\ref{alg:simple}, we can show that Algorithm~\ref{alg:sparse} finds the lexicographically \emph{maximum} fractional parity base $y^*$.
This follows from the next lemma for \textsc{BuildFractionalParityBase}.

\begin{lemma}\label{lem:BFMM}
     Given $L = \{a, \dots, b\}$ and $y$ such that $y_i = y_i^*$ for $i < a$ and $y_i = 0$ for $i \geq a$, then \textsc{BuildFractionalParityBase} outputs $y$ such that $y_i = y_i^*$ for $i \leq b$ and $y_i = 0$ for $i > b$.
\end{lemma}
\begin{proof}
    The proof is by induction on the size of $L$.

    First, suppose that $L = \{a\}$.
    We consider three cases based on the value of $y_a^*$.

    \paragraph{Case 1: $y_a^* = 0$.} Then $Z\bigl(y + \frac{1}{2}e_a\bigr)$ is singular; otherwise, by \Cref{lem:characterize-half-integral-matching,lem:extensible}, there exists a fractional parity base $z \geq y$ such that $z_a \geq \frac12 > y_a^*$, which contradicts the lexicographical maximality of $y^*$.
    Hence $y_a = 0$ by Line~\ref{line:BFMM-x-update}.
    Note that $Z\bigl(y + \frac{1}{2}e_a\bigr)$ is nonsingular if and only if $\det\bigl(I + D_{\frac12}[L]M[L]\bigr) \neq 0$ by \Cref{lem:small-update}.

    \paragraph{Case 2: $y_a^* = \frac12$.} Then $Z\bigl(y + \frac{1}{2}e_a\bigr)$ is nonsingular by \Cref{lem:extensible} with $y = y^*$. Furthermore, $Z(y + e_a)$ is singular by a similar argument.
    Hence, $y_a = \frac12$.

    \paragraph{Case 3: $y_a^* = 1$.} Then $Z(y + e_a)$ is nonsingular by \Cref{lem:extensible} with $y = y^*$, so $y_a = 1$.

    Thus, we have shown $y_a = y_a^*$.

    Next, suppose that $|L| > 2$.
    Then, the algorithm recurses on $L_1 = \bigl\{a, \dots, \floor{\frac{a+b}{2}}\bigr\}$ and $L_2 = \bigl\{\floor{\frac{a+b}{2}} + 1, \dots, b\bigr\}$.
    By the induction hypothesis, we have $y_i = y_i^*$ for $i \leq \floor{\frac{a+b}{2}}$ and $y_i = 0$ for $i > \floor{\frac{a+b}{2}}$ after Line~\ref{line:BFMM-L1-update}.
    So the new $y$ satisfies the assumption for $L_2$.
    Applying the induction hypothesis again, we have $y_i = y_i^*$ for $i \leq b$ and $y_i = 0$ for $i > b$ in the final output, which completes the proof.
\end{proof}

We have the following theorem.
\begin{theorem}\label{thm:sparse}
   If $|R| \geq 16mn$, \Cref{alg:sparse} finds the lexicographically maximum fractional parity base $y^*$ in $O(m^\omega)$ time with probability at least $\frac12$.
\end{theorem}
\begin{proof}
The output equals $y^*$ by the above lemma for $a=1$, $b=m$, and $y = \bfzero$.
Let us analyze the time complexity of \Cref{alg:sparse}.
Let $f(m)$ be the time complexity of \textsc{BuildFractionalParityBase} with $|L| = m$.
Then, $f(1) = O(1)$ and $f(m) = 2f\bigl(\frac{m}{2}\bigr) + O(m^\omega)$.
Solving this recursion, $f(m) = O(m^\omega)$.
Computing ${Z(\bfzero)}^{-1}$ takes $O({(n + m)}^\omega) = O(m^\omega)$ time as $n \leq 2m$.
In total, the time complexity of \Cref{alg:sparse} is $O(m^\omega)$.
The probability of success is similar to \Cref{alg:simple}, so we omit it.
\end{proof}

\subsection{An \texorpdfstring{$O(mn^{\omega - 1})$}{O(mn\textasciicircum(omega - 1))}-Time Algorithm}
\Cref{alg:sparse} is faster than \Cref{alg:simple} for large $n$, e.g., $n = \Omega(m)$.
In this section, we give an $O(mn^{\omega - 1})$-time algorithm based on \Cref{alg:sparse}, which is faster for all $n$.

Without loss of generality, we assume that $m$ is an integer multiple of $n$.
If this is not the case, we can copy lines to satisfy the condition.
This does not change the optimal value.

The algorithm splits the set $L$ of lines into $\frac{m}{n}$ subsets $L_1, \dots, L_{\frac{m}{n}}$ of size $n$.
Then, we apply \textsc{BuildFractionalParityBase} to $L_1, \dots, L_m$ sequentially.
The algorithm is shown in \Cref{alg:faster}.

\begin{algorithm}[tb]
    \caption{An $O(mn^{\omega-1})$-time algorithm to find a fractional parity base}\label{alg:faster}
    \begin{algorithmic}[1]
        \Input{A fractional matroid parity polytope $P$ given as $a_i, b_i\in \K^n$ ($i \in [m]$) and a finite subset $R$ of $\K$.}
        \Output{The lexicographically maximum fractional parity base in $P$}
        \State $y \gets \bfzero$ and $L \gets [m]$.
        \State Compute $T := T(\bfzero)[L]$ with random substituted values from $R$.
        \State Compute $M := (I_2 \otimes B) T^{-1} {(I_2 \otimes B)}^\top$ and $M^{-1}$ in $O(mn^{\omega-1})$ time as shown in \Cref{claim:compute-Shur}.
        \State Partition $[m]$ intro $\frac{m}{n}$ subsets $L_1, \dots, L_{\frac{m}{n}}$ with each size $n$.
        \For{$\ell = 1, \dots, \frac{m}{n}$}
            \State Compute $M := {Z(y)}^{-1}[L_\ell]$ in $O(n^\omega)$ time as shown in \Cref{claim:compute-Zsub}.
            \State $x \gets \textsc{BuildFractionalParityBase}(L_\ell, y, M)$
        \EndFor
        \State \Return $x$
    \end{algorithmic}
\end{algorithm}

We have the following theorem.
\begin{theorem}\label{thm:faster}
    If $|R| \geq 16mn$, \Cref{alg:faster} finds the lexicographically maximum fractional parity base in $O(mn^{\omega-1})$ time with probability at least $\frac12$.
\end{theorem}

We first show that the matrix computation in the algorithm can be done in the claimed running time.

\begin{claim}\label{claim:compute-Shur}
    We can compute $M = (I_2 \otimes B) T^{-1} {(I_2 \otimes B)}^\top$ and $M^{-1}$ in $O(mn^{\omega-1})$ time.
\end{claim}
\begin{proof}
    To compute $M = (I_2 \otimes B) T {(I_2 \otimes B)}^\top$, note that $T$ is a block diagonal matrix with $4\times 4$ blocks.
    Hence, $T^{-1} {(I_2 \otimes B)}^\top$ can be computed in $O(mn)$ time.
    Then, we can compute the matrix product of $I_2 \otimes B$ and $T^{-1} {(I_2 \otimes B)}^\top$ in $O(mn^{\omega-1})$ time.
    Similarly, $M^{-1}$ can be computed in $O(mn^{\omega -1})$ time.
\end{proof}

Next, we consider the $\ell$th iteration of the for loop.
Let $Z := Z(\bfzero)$.
\begin{claim}
    Given $M^{-1}$,
    we can compute $Z[I, J]$ in $O(n^\omega)$ time for any $I, J$ with $|I|, |J| \leq 4n$.
\end{claim}
\begin{proof}
    By~\eqref{eq:Schur-inverse},
    \begin{align}
    Z^{-1}[I,J]
    = T^{-1}[I,J] - {\bigl(T^{-1} (I_2 \otimes B)\bigr)}^\top [I, *] \cdot M^{-1} \cdot \bigl((I_2 \otimes B)T^{-1}\bigr)[*, J].
    \end{align}
    Since ${(T^{-1} (I_2 \otimes B))}^\top [*,I]$ and $((I_2 \otimes B)T^{-1})[*, J]$ are matrices of size at most $2n\times 2n$ computable in $O(n^2)$ time and $M^{-1}$ is a known $2n \times 2n$ matrix, we can compute the product in $O(n^\omega)$ time.
\end{proof}

\begin{claim}\label{claim:compute-Zsub}
    In each iteration, ${Z(y)}^{-1}[L_\ell]$ can be computed in $O(n^\omega)$ time.
\end{claim}
\begin{proof}
    $Z(y)$ is identical to $Z$ except $Z(y)[S] \neq Z[S]$ for $S := \supp(y)$.
    Note that $|S| \leq 2|x| \leq n$ as $x \in P$.
    By \Cref{lem:small-update}, ${Z(y)}^{-1}[L_\ell]$ can be computed in $O(n^\omega)$ time, given $Z^{-1}[L_\ell]$, $Z^{-1}[S, L_\ell]$, $Z^{-1}[L_\ell, S]$, and $Z^{-1}[L_{\ell-1}]$.
    These four matrices are computable in $O(n^\omega)$ time by the previous claim.
\end{proof}

Therefore, each step can be carried out in the claimed time complexity.
We now prove \Cref{thm:faster}.

\begin{proof}[Proof of~\Cref{thm:faster}]
    Using \Cref{lem:BFMM} repeatedly, we can check that it outputs $y^*$.
    So it suffices to show the time complexity.
    By \Cref{thm:sparse}, \textsc{BuildFractionalParityBase} for each $L_\ell$ takes $O(n^\omega)$ time since $|L_\ell| = n$.
    In total, \Cref{alg:faster} runs in $\frac{m}{n} \cdot O(n^\omega) = O(mn^{\omega-1})$ time.
    The probability of success is similar to \Cref{alg:simple,alg:sparse}, so we omit it.
\end{proof}

\subsection{Extension and Bit Complexity}\label{subsec:extension}
\subsubsection{An \texorpdfstring{$O(mn^{\omega - 1})$}{O(mn\textasciicircum(omega - 1))}-time algorithm for finding a maximum fractional matroid matching}
So far, we described our faster algorithm for finding a fractional parity base.
We can easily modify it to find a maximum fractional matroid matching as follows.

First, we compute the Schur complement $M = (I_2 \otimes B)T^{-1}{(I_2 \otimes B)}^\top$ in $O(mn^{\omega - 1})$ time as in \Cref{alg:faster}.
Then, we find a maximum nonsingular principal submatrix $M[I]$ with $|I| = 4k$ in $O(n^\omega)$ time, where $k$ is the cardinality of a maximum fractional matroid matching.
Note that $M$ is a $2n \times 2n$ skew-symmetric matrix.
By a similar argument as in \Cref{lem:extensible}, we can show that $Z(y)[I]$ is nonsingular if $y$ is extensible to a maximum fractional matroid matching.
Hence, we can restrict our attention to the principal matrix $Z(y)[I]$ rather than the entire matrix $Z(y)$.
In other words, we run \Cref{alg:faster} with
\begin{align}
    Z'(y) =
    \begin{bNiceArray}[margin]{c|ccc}
        O                              & (I_2 \otimes B_1)[I, *] & \Cdots & (I_2 \otimes B_m)[I, *] \\ \midrule
        -{(I_2 \otimes B_1)}^\top[*,I] & Y_1 \otimes \Delta &  &  \\
         \Vdots                        &  & \Ddots   & \\
        -{(I_2 \otimes B_m)}^\top[*,I] &  &  & Y_m \otimes \Delta
    \end{bNiceArray}.
\end{align}
The time complexity is still $O(mn^{\omega-1})$.

\subsubsection{Las Vegas algorithm}
So far, we devised Monte Carlo algorithms; they always run in polynomial time and find a correct solution with probability at least a constant.
We can easily make them Las Vegas---algorithms that always find a correct solution (if they terminate) and the expected time-complexity is polynomial---as follows.
Run \Cref{alg:faster,alg:2-cover} independently and obtain a maximum fractional matroid matching $y$ and the dominant 2-cover $(S^*, T^*)$.
If $2|y| = \dim S^* + \dim T^*$, then they are optimal primal-dual solutions by \Cref{thm:CLV} and return them.
Otherwise, rerun the algorithms.
Both algorithms succeed with probability at least $\frac{1}{2}$, so the expected number of iterations until both succeed is $O(1)$.
As the time-complexity of each iteration is $O(mn^{\omega+1})$, this yields an $O(mn^{\omega+1})$-time Las Vegas algorithm for finding a maximum fractional linear matroid matching and the dominant 2-cover.

\subsubsection{Bit complexity on rational field}\label{sec:bit-complexity}
We can bound the bit complexity of intermediate numbers of matrices appearing in \Cref{alg:faster} when we run it on the rational field $\Q$.
More precisely, we can show the bit complexity is $O((m+n)\max\{\log(m+n), \agbr{B}\})$, where $\agbr{B}$ denotes the bit complexity of input matrices $B_1, \dots, B_m \in \Q^{n \times 2}$.
Note that this is significant because it is unknown (to the best of our knowledge) whether the known algorithm~\cite{Chang2001a} for fractional matroid matching has polynomial bit complexity.

By the Schwartz-Zippel lemma, it suffices to substitute random integrals from $R = \{0, \dotsc, 16mn-1\}$ to indeterminates so that the algorithm finds a solution with probability at least $\frac{1}{2}$.
For any $y$, the resulting matrix $Z(y)$ contains either an integer with magnitude at most $256m^2n^2$ or an entry from $B_i$ in its entries.
So the bit complexity of any entry of $Z(y)$ is $O(\max\{\log(mn), \agbr{B}\})$.
Hence, by Cramer's rule, its inverse ${Z(y)}^{-1}$ has rational entries with bit complexity $O((m+n)\log(m+n) + (m+n)\max\{\log(mn), \agbr{B}\}) = O((m+n)\max\{\log(mn), \agbr{B}\})$.
The other matrices such as the Schur complement have the asymptotically same bit complexity.

\subsubsection{Weighted fractional linear matroid matching}
Our faster algorithm for the fractional linear matroid parity problem can be used to obtain a faster algorithm for weighted fractional linear matroid parity base via the primal-dual framework of \cite{Gijswijt2013}.
Let $w \in \R^m$ be a nonnegative weight.
The goal of the weighted fractional linear matroid parity problem is to find a fractional parity base $y$ that maximizes $w^\top y$.
Gijswijt et al.~\cite{Gijswijt2013} showed that a weighted problem can be reduced to a sequence of unweighted problems.
Specifically, given an unweighted algorithm that returns a maximum fractional matroid matching and the dominant 2-cover, a weighted problem can be solved by calling the unweighted algorithm $O(n^3)$ times.
They used an unweighted algorithm of \cite{Chang2001a} and obtained an $O(\GPtime)$ time algorithm.
For the linear case, we can simply replace the unweighted algorithm with \Cref{alg:faster,alg:2-cover} and obtain an $O(mn^{\omega+4})$-time randomized algorithm.
Recall that the probability that our algorithms fail is at most $\frac{8mn}{|R|}$ when random values are drawn from $R \subseteq \K$.
By the union bound, the probability that any of $O(n^3)$ executions fails is at most $\frac{O(mn^4)}{|R|}$.
Hence, taking $R$ of size $O(mn^4)$, the algorithm finds a maximum weight fractional parity base with probability at least $\frac{1}{2}$.

\subsubsection{Hidden fractional matroid parity}
Finally, we mention that our algorithm can be used to solve the ``hidden fractional linear matroid parity problem'', which is to compute the nc-rank of a matrix space $\mathcal{A} = \agbr{A_1, \dotsc, A_m} \le \K^{n \times n}$ with the promise that it has a rank-two skew-symmetric basis $a_1 \wedge b_1, \dotsc, a_m \wedge b_m$.
This is a generalization of the ``hidden linear matroid intersection problem'' considered in~\cite{Gurvits2004,Ivanyos2015}.

As a byproduct of our algorithm, we have shown in \Cref{thm:blowup-size-2} that $\ncrank \mathcal{A} = \frac12 \rank \blowup{A'}{2}$ holds with $A' = \sum_{i=1}^m x_i (a_i \wedge b_i$), i.e., the second-order blow-up is sufficient for a linear matrix with rank-two skew-symmetric coefficients.
Furthermore, the blow-up size is independent of the choice of a basis of a matrix space.
Therefore, $\ncrank \mathcal{A} = \modify{\frac12} \ncrank \blowup{A}{2}$ holds as well, where $A = \sum_{i=1}^m x_i A_i$ is the linear matrix in the given basis.
Thus, $\ncrank \mathcal{A}$ can be computed in $O(mn^2 + n^\omega)$ time by random substitution to $\blowup{A}{2}$.
One can also devise deterministic algorithms by computing the nc-rank of $\mathcal A$ with the known deterministic algorithms. We are not sure whether these algorithms can be made faster for matrix spaces with a rank-two skew-symmetric basis using our results; we leave it as future work.

\section*{Acknowledgments}
We thank anonymous reviewers for their helpful comments.
This work was supported by JST ERATO Grant Number JPMJER1903, JST CREST Grant Number JPMJCR24Q2, JST PRESTO Grant Number JPMJPR24K5, JST FOREST Grant Number JPMJFR232L, JST ACT-I Grant Number JPMJPR18U9, JSPS KAKENHI Grant Numbers JP19K20212, JP22K17853, and 24K21315, and JSPS Overseas Research Fellowships.

\printbibliography[heading=bibintoc]

\end{document}